\documentclass{amsart}
\usepackage{color}
\usepackage{latexsym}
\usepackage{amsmath}%
\usepackage{amsfonts}%
\usepackage{amssymb}%
\usepackage{amsthm}
\usepackage{graphicx}
\usepackage[english]{babel}
\usepackage[latin1]{inputenc}
\usepackage[T1]{fontenc}
\usepackage{enumerate} 
\newtheorem{theorem}{Theorem}
\theoremstyle{plain}

\newtheorem{lemma}{Lemma}

\newtheorem{remark}{Remark}

\numberwithin{equation}{section} 
\let\dsp=\displaystyle
\newcommand\beq{\begin{equation}}
\newcommand\eeq{\end{equation}}
\renewcommand{\emph}{\textbf}
\newcommand{\brk}[1]{\left(#1\right)}          
\newcommand{\Brk}[1]{\left[#1\right]}          

\newcommand{\x}{\boldsymbol{x}}
\newcommand{\bolda}{\boldsymbol{a}}

\newcommand{\e}{\epsilon}
\newcommand{\D}{\mathcal{D}}
\newcommand{\f}{\boldsymbol{f}}
\newcommand{\I}{\boldsymbol{I}}

\newcommand{\bn}{{\boldsymbol{n}}}
\renewcommand{\to}{\rightarrow}

\newcommand{\str}{{\boldsymbol{\tau}}}
\newcommand{\strs}{{\boldsymbol{\sigma}}}
\newcommand{\bu}{{\boldsymbol{u}}}
\newcommand{\bN}{{\boldsymbol{N}}}

\newcommand{\be}{\boldsymbol{e}}

\newcommand{\gbu}{{\grad\bu}}
\newcommand{\Dbu}{{\boldsymbol{D}(\bu)}}

\newcommand{\bs}{\boldsymbol{s}}
\newcommand{\cs}{\mathcal{S}}
\newcommand{\F}{\mathcal{F}}
\newcommand{\G}{\mathcal{G}}

\renewcommand{\div}{\operatorname{div}}

\newcommand{\tr}{\operatorname{tr}}

\newcommand{\grad}{\boldsymbol{\nabla}}

\newcommand{\R}{\mathbb{R}}
\newcommand{\Z}{\mathbb{Z}}
\newcommand{\RS}{\mathbb{R}^{d \times d}_S}

\newcommand{\deriv}[2]{\frac{d#1}{d#2}}
\newcommand{\pd}[2]{\frac{\partial#1}{\partial#2}}

\newcommand{\pds}[1]{\partial_{#1}}
\newcommand{\intd}{\int_{\D_t}}
\usepackage{verbatim}
\newcommand{\X}{{\boldsymbol{X}}}
\begin{document}

\title[A new model for shallow viscoelastic fluids]{A new model for shallow viscoelastic fluids}
\author{Fran\c{c}ois Bouchut}
\author{S\'ebastien \textsc{Boyaval}}
\address{\textbf{CNRS \& Universit\'e Paris-Est}, Laboratoire d'Analyse et de Math\'ematiques Appliqu\'ees, Universit\'e Paris-Est - Marne-la-Vall\'ee, 5 boulevard Descartes, Cit\'e Descartes - Champs-sur-Marne, 77454 Marne-la-Vall\'ee cedex 2 - France}
\email{francois.bouchut@univ-mlv.fr }
\urladdr{http://perso-math.univ-mlv.fr/users/bouchut.francois/}
\address{\textbf{Universit\'e Paris-Est}, Laboratoire d'hydraulique Saint Venant ( EDF R \& D -- Ecole des Ponts ParisTech -- CETMEF ), EDF R \& D 6 quai Watier, 78401 Chatou Cedex, France and \\ \textbf{INRIA}, MICMAC Project, Domaine de Voluceau, BP. 105 - Rocquencourt, 78153 Le Chesnay Cedex, France} 
\email{sebastien.boyaval@enpc.fr}
\urladdr{http://cermics.enpc.fr/~boyaval/}
\thanks{This work was completed while SB was visiting MATHICSE -- ASN chair at EPFL. SB would like to thank Marco Picasso and Jacques Rappaz for their kind hospitality.}
\date{\today}

\begin{abstract}
We propose a new reduced model for gravity-driven free-surface flows of shallow viscoelastic fluids. It is obtained by an asymptotic expansion of the upper-convected Maxwell model for viscoelastic fluids. The viscosity is assumed small (of order epsilon, the aspect ratio of the thin layer of fluid), but the relaxation time is kept finite. Additionally to the classical layer depth and velocity in shallow models, our system describes also the evolution of two components of the stress. It has an intrinsic energy equation. The mathematical properties of the model are established, an important feature being the non-convexity of the physically relevant energy with respect to conservative variables, but the convexity with respect to the physically relevant pseudo-conservative variables. Numerical illustrations are given, based on a suitable well-balanced finite-volume discretization involving an approximate Riemann solver.
\end{abstract}

\maketitle
\tableofcontents

\section{Introduction: thin layer approximations of non-Newtonian flows}
\label{sec:intro}

There are many occurences of free-surface non-Newtonian flows over an inclined topography in nature, for instance geophysical flows: mud flows, landslides, debris avalanches \ldots. Their mathematical prediction is important, typically for safety reasons in connection with land use planning in the case of geophysical flows. But their modelling is still difficult, as one can conclude from the continuing intense activity in that area (see the reviews~\cite{ancey-2007,matar-kraster-2009} e.g. plus the numerous references cited therein and below). In this paper, our purpose is (i) to derive a new simple model for a thin layer of viscoelastic non-Newtonian fluid over a given topography at the bottom when the motion is essentially driven by gravity forces and (ii) to numerically investigate the prediction of that simple model in benchmark cases. 

Our methodology follows the standard derivation of the Saint-Venant model for gravity-driven shallow water flows, as developped in~\cite{gerbeau-perthame-2001} for instance. For non-Newtonian fluids, there already exist similar projects in the literature. But to our knowledge, they use different models as starting point: power-law and Bingham models~\cite{fernandez-nieto-noble-vila-2010,bresch-fernandeznieto-ionescu-vigneaux-2010}, or a kinetic model for microscopic FENE dumbbells~\cite{narbona-reina-bresch-2010} (see also the Section~\ref{sec:interpretationmicro} for comparison with a kinetic interpretation of our model using Hookean dumbbells). Here, we derive a reduced form of the \textsl{Upper-Convected Maxwell} (UCM) equations, a widely-used differential model for \textsl{viscoelastic} fluids valid in generic geometries, in the specific case of gravity-driven free-surface thin-layer flows over an inclined topography. In particular, the influence of each term in the equations is compared with the aspect ratio
\begin{equation}
 h/L \approx \e \ll 1
\label{eq:aspectratio}
\end{equation}
between the layer depth $h$ and its longitudinal characteristic length $L$ as a function of a small parameter $\e$.  

The reduced model obtained is computationally much less expensive to solve numerically than the full UCM model with an unknown free surface (compare for instance with numerical simulations in~\cite{pasquali-scriven-2002} of a full 3D model). So we can easily investigate its predictions in a number of test cases, to show the capabilities of the model. We note two important aspects from the mathematical viewpoint to discretize our model. It is endowed with a natural energy law (inherited from the UCM model) but has a non-standard hyperbolic structure (the physically relevant energy is not convex with respect to the conservative variables). These features of our model have important consequences on the numerical simulation. Whereas we can only perform numerical simulations in a formal way (because the non-standard hyperbolic structure does not fit in the usual numerical analysis), we can nevertheless confirm that they are physically meaningful (owing to the natural energy law, satisfied at the discrete level).

Regarding the literature, we would like to make two further comments in order to better situate our work and its originality.
On the one hand, numerous models for thin layers of non-Newtonian fluids have already been derived in the physics literature. We are aware of only one reduced version of the UCM model which is very close to ours, see~\cite{entov-yarin-1984,entov2006dynamics} and a sketch of that work in~\cite{renardy-2000}. But the reduced model, obtained with another methodology and with a different perspective (ad-hoc model to investigate the break-up and swell of free jets and thin films rather than asymptotic analysis of general fluid equations), finally applies in different conditions (without gravity and topography). The other models we are aware of, typically obtained either with a different methodology or (sometimes and) in different conditions (like spin coating e.g.), 
are different, see for instance~\cite{spaid-homsy-1994,kang-chen-1995,ro-homsy-1995,kalliadasis-bielarz-homsy-2000,li-luo-qi-zhang-2011}.
On the other hand, recent works in the mathematical literature also studied reduced models for thin layers of viscoelastic flows. For instance~\cite{bayada-chupin-martin-2007,bayada-chupin-grec-2009} derive reduced models for the \textsl{Oldroyd-B} (OB) system of equations, where a purely viscous component $\div(\eta_s\Dbu)$ is added to the stress term in the right-hand side of~(\ref{eq:ns}) in comparison with the UCM equations. But our project is different in essence from the thin layer models obtained for those viscoelastic flows without free surface and essentially driven by viscosity instead of gravity. Recall that here we focus on gravity-driven shallow regimes, and that is why we consider the UCM model in particular rather than the OB model (the viscosity only plays a minor role here).

In Section~\ref{sec:ucm} below, we recall the UCM model for viscoelastic fluids and some of its properties in the mathematical setting that is adequate to our model reduction. Then our new reduced model is derived in Section~\ref{sec:derivation} under a given set of clear mathematical hypotheses (which we classically cannot embed into an existence theory for solutions to the non-reduced UCM system of equations). Section~\ref{sec:properties} is devoted to the study of some mathematical properties of our new reduced model. In Section~\ref{sec:numerics}, we provide numerical simulations in benchmark situations where shallow viscoelastic flows could be advantageously modelled by our new system of equations. Last, in Section~\ref{sec:conclusion}, a physical interpretation of situations modelled by our system of equations is given in conclusion, along with threads for next studies.

\section{Mathematical setting with the Upper-Convected Maxwell model for viscoelastic fluids}
\label{sec:ucm}

The evolution for times $t\in[0,+\infty)$ of the flow of a given portion of some viscoelastic fluid confined in a moving domain $\D_t\subset\R^d$ ($d=2$ or $3$) with piecewise smooth boundary $\partial\D_t$ is governed by the following set of equations, the so-called \textsl{Upper-Convected Maxwell} (UCM) model~\cite{barnes-hutton-walters-1989,bird-curtiss-armstrong-hassager-1987a,renardy-2000}:
\begin{equation}
	\div\bu= 0 \qquad \mbox{in } \D_t ,
	\label{eq:divu}
\end{equation}
\begin{equation}
	\pds{t}\bu + (\bu\cdot\grad)\bu= -\grad p + \div\str + \f
	\qquad \mbox{in } \D_t,
	\label{eq:ns}
\end{equation}
\begin{equation}
	\pds{t}\str+(\bu\cdot\grad)\str = (\gbu)\str+\str(\gbu)^T
	+ \frac{1}{\lambda}\brk{\eta_p\Dbu-\str}
	\qquad\mbox{in } \D_t,
	\label{eq:ucm}
\end{equation}
where:
\begin{itemize}
 \item $\bu : (t,\x)\in[0,+\infty)\times\D_t \mapsto \bu(t,\x)\in\R^d$ is the velocity of the fluid,
 \item $\Dbu : (t,\x)\in[0,+\infty)\times\D_t \mapsto \Dbu(t,\x)\in\RS$, where $\RS$ denotes symmetric real $d \times d$ matrices, is the rate-of-strain tensor linked to the fluid velocity $\bu$ through the relation
\begin{equation}
	\Dbu=\frac12(\gbu+\gbu^T).
	\label{eq:D(u)}
\end{equation}
 \item $p : (t,\x)\in(0,+\infty)\times\D_t \mapsto p(t,\x)\in\R$ is the pressure, 
 \item $\str : (t,\x)\in[0,+\infty)\times\D_t \mapsto\str(t,\x)\in \RS$ is the symmetric extra-stress tensor,
 \item $\eta_p,\lambda>0$ are physical parameters, respectively a viscosity only due to the presence of elastically deformable particles in the fluid, and a relaxation time corresponding to the intrinsic dynamics of the deformable particles,
 \item $\f: (t,\x)\in [0,+\infty)\times\D_t \mapsto \f(t,\x)\in\R^d$ is a body force.
\end{itemize}
Notice that we have assumed the fluid {\sl homogeneous} (with constant mass density, hence normalized to one). We also refer to the Section~\ref{sec:conclusion} where more details about the UCM model are given along with a physical interpretation of our results. From now on, we assume translation symmetry ($d=2$), we endow $\R^2$ with a cartesian frame $(\be_x,\be_z)$ such that $\f \equiv -g \be_z$ corresponds to gravity and we assume that $\D_t$ has the following geometry (in particular, surface folding like in the case of breaking waves is not possible):
\begin{equation}
	\forall t\in[0,+\infty)\,, \quad \x=(x,z)\in\D_t \Leftrightarrow x\in(0,L), \quad  0 < z - b(x) < h(t,x),
	\label{eq:domain}
\end{equation}
where $b(x)$ is the topography elevation and $b(x)+h(t,x)$ is the \textsl{free surface} elevation of our thin layer of fluid. Note that the width $h(t,x)$ is an unknown of the problem (it is a free boundary problem). We shall denote as $a_x$ (respectively $a_z$) the component in direction $\be_x$ (resp. $\be_z$) of any vector (that is a rank-1 tensor) variable $\bolda$, and similarly the components of higher-rank tensors : $a_{xx}, a_{xz},\ldots$ We denote by $\bn:x\in(0,L)\to\bn(x)$ the unit vector of the direction normal to the bottom and
inward the fluid:
\begin{equation}
	n_x = \frac{-\pds{x}{b}}{\sqrt{1+(\pds{x}{b})^2}} \qquad n_z = \frac{1}{\sqrt{1+(\pds{x}{b})^2}} \,.
	\label{eq:normal}
\end{equation}
We supply the UCM model with boundary conditions for all $t\in (0,+\infty)$: pure slip at bottom,
\begin{equation}
	\bu\cdot\bn=0,
	\qquad\mbox{for }z=b(x),\quad x\in(0,L),
	\label{eq:pureslip}
\end{equation}
\begin{equation}
	\str\bn = \left((\str\bn)\cdot\bn\right) \bn,
	\qquad\mbox{for }z=b(x),\quad x\in(0,L),
	\label{eq:nofriction}
\end{equation}
kinematic condition at the free surface $N_t+\bN\cdot\bu=0$ where $(N_t,\bN)$ is the time-space normal, i.e.
\begin{equation}
	\pds{t}h + u_x\pds{x}(b+h)= u_z,
	\qquad\mbox{for }z=b(x)+h(t,x),\quad x\in(0,L),
	\label{eq:kinematic}
\end{equation}
no tension at the free surface,
\begin{equation}
	(p\I-\str)\cdot(-\pds{x}(b+h),1)= 0,
	\qquad\mbox{for }z=b(x)+h(t,x),\quad x\in(0,L),
	\label{eq:notension}
\end{equation}
plus (for example) inflow/outflow boundary conditions or periodicity in $x$.
We insist on~\eqref{eq:nofriction} without friction. Adding a friction term in~\eqref{eq:nofriction} would not yield the same result. 
Finally, the Cauchy problem is supplied with initial conditions 
\begin{equation}
	\bu(0,\x) = \bu^0(\x), 
	\qquad \str(0,\x) = \str^0(\x),
	\qquad h(0,x)=h^0(x),
	\label{eq:str0}
\end{equation}
assumed {\it sufficiently smooth} for a solution to exist. Note indeed that the existence theory for solutions to the UCM system~(\ref{eq:divu}--\ref{eq:ns}--\ref{eq:ucm}) is still very limited (see {\it e.g.}~\cite{joseph-renardy-saut-1985,joseph-saut-1986}), like for non-Newtonian flows with a free surface (see {\it e.g.}~\cite{lemeur-1995,lemeur-2011} for the so-called Oldroyd-B model with a viscous term in~\eqref{eq:ns}).
                                                                                                                                                                                                
Last, we recall some essential features of the UCM model~(\ref{eq:divu}--\ref{eq:str0}). Let $\strs : (t,\x)\in[0,+\infty)\times\D_t \mapsto\strs(t,\x)\in \RS$ be the symmetric conformation tensor linked to the symmetric extra-stress tensor $\str$ through the relation
\begin{equation}
	\strs = \I + \frac{2\lambda}{\eta_p}\str,
	\label{eq:defsigma}
\end{equation}
where $\I$ denotes the $d$-dimensional identity tensor. The UCM model can be written using the variable $\strs$ instead of $\str$. Indeed, $\frac{\eta_p}{2\lambda}\div\strs$ replaces $\div\str$ in~\eqref{eq:ns}, and~\eqref{eq:ucm} should be replaced with
\begin{equation}
	\pds{t}\strs+(\bu\cdot\grad)\strs = (\gbu)\strs+\strs(\gbu)^T
	+ \frac{1}{\lambda}\brk{\I-\strs}
	\qquad\mbox{in } \D_t\,.
\end{equation}
In addition, the following properties are easily derived following the same steps as in~\cite{boyaval-lelievre-mangoubi-2009} for the Oldroyd-B model (except for the absence of the dissipative viscous term $\eta_s|\Dbu|^2$). First, for physical reasons, $\strs$ should take only positive definite values (this is easily deduced when $\strs$ is interpreted as the Grammian matrix of stochastic processes, see~\cite{lelievre-2004} e.g. and Section~\ref{sec:interpretationmicro}). The initial condition~(\ref{eq:str0}) should thus be chosen so that $\strs(t=0)$ is positive definite. Provided the system~(\ref{eq:divu}--\ref{eq:str0}) has sufficiently smooth initial conditions and the velocity field $\bu$ remains sufficiently smooth, $\strs$ indeed remains positive definite (see~\cite{boyaval-lelievre-mangoubi-2009} e.g.; where the viscosity $\eta_s$ plays no role in the proof). Second, the system~(\ref{eq:divu}--\ref{eq:str0}) is endowed with an energy (the physical free energy)
\beq
\label{eq:E}
F(\bu,\str) = \intd \left( \frac12 |\bu|^2 + \frac{\eta_p}{4\lambda} \tr(\strs-\ln\strs-\I) - \f\cdot\x \right) d\x
\eeq
which, following Reynolds transport formula and~\cite{boyaval-lelievre-mangoubi-2009}, is easily shown to decay as
\beq
\label{eq:Edecay}
\deriv{}{t}  F(\bu,\str) = - \frac{\eta_p}{4\lambda^2} \intd \tr(\strs+\strs^{-1}+2\I) d\x \,.
\eeq

\section{Formal derivation of a thin layer approximation}
\label{sec:derivation}
                                                                                                                                                                                                                                                                                          
Our goal is to derive a reduced model approximating~(\ref{eq:divu}--\ref{eq:str0}) in the thin layer regime $ h\ll L $ where $L$ is a characteristic length of the flow. We follow the formal approach of~\cite{bouchut-mangeneycastelnau-perthame-vilotte-2003,bouchut-westdickenberg-2004}. Our main assumption is thus $h/L = O(\e)$, thereby introducing an adimensional parameter $\e\to0$. In the following, we simply write
$$ \text{(H1) } h = O(\e), $$
where one should infer the unit ($L$) and the limit ($\e\to0$). This corresponds to 
a rescaling of the space coordinates 
with an aspect ratio $\e$ between the vertical and horizontal dimensions.
We shall also implicitly use a characteristic time $T$ over which the physical variables vary significantly.
Then, our task can be formulated as: find a set of \textit{non-negative integers}
$$
I=(I_{u_x},I_{u_z},I_p,I_{\tau_{xx}},I_{\tau_{xz}},I_{\tau_{zz}}) 
$$
such that a closed system of equations for variables $(\tilde\bu,\tilde p,\tilde\str)$ approximating $(\bu,p,\str)$ 
holds and
\begin{equation}
\label{eq:scaling}
(\bu-\tilde\bu,p-\tilde p,\str-\tilde\str) = O(\e^I) 
\end{equation}
is a uniform approximation on $\D_t$ (where the powers are applied componentwise). Note that in~\eqref{eq:scaling} and all along the paper, $O$ has to be understood componentwise in the unit corresponding to the variable. In particular, the previous assumption~\eqref{eq:scaling} more explicitly signifies
$$
\left(\frac{T}{L}(\bu-\tilde\bu),\frac{T^2}{L^2}(p-\tilde p),\frac{T^2}{L^2}(\str-\tilde\str)\right) = O(\e^I) \,.
$$
We proceed heuristically, increasing little by little the degree of our assumptions on $I$. Hopefully, the reduced model found that way corresponds to a physically meaningful regime.

We recall that another viewpoint 
is to find a closed system of equations for depth-averages of the main variables of the system~(\ref{eq:divu}--\ref{eq:str0}). The link with our approach is as follows. The conservation of mass for an incompressible, inviscid fluid governed by~(\ref{eq:divu}) within a control volume governed by the evolution of the free-surface height as given by the kinematic boundary condition~(\ref{eq:kinematic}) and the boundary condition~(\ref{eq:pureslip}) at the bottom reads as an evolution equation for the free-surface height $h$ (using the Leibniz rule) where the depth-averaged velocity profile $\tilde u_x := \frac1h\int_b^{b+h}u_x dz$ enters,
\begin{equation}
\label{eq:h1}
\forall t,x\in[0,+\infty)\times(0,L) \quad 
0 = \int_b^{b+h}(\pds{x}u_x + \pds{z}u_z) dz = \pds{t}h + \pds{x}\left(\int_b^{b+h}u_x dz\right) = \pds{t}h + \pds{x}(h\tilde u_x)\,.
\end{equation}
The challenge in the derivation of a reduced model for thin layers is then to find a closure for the evolution of $\tilde u_x$
in terms of the variables $(\tilde\bu,\tilde p,\tilde\str)$. In particular, depth-averaging the equation for $u_x$ with the boundary conditions~(\ref{eq:kinematic}--\ref{eq:nofriction}--\ref{eq:notension}) 
gives, using again the Leibniz rule,
\begin{equation}
\label{eq:avux}
\pds{t}\left(\int_b^{b+h}u_x\,dz\right)
+ \pds{x}\left(\int_b^{b+h}\left(u_x^2+p-\tau_{xx}\right)\,dz\right)
 = \left[(\tau_{xx}-p)\pds{x}b-\tau_{xz}\right]|_b,
\end{equation}
showing that a typical problem is to write an approximation for $\int_b^{b+h}u_x^2$ and for the source term in the right-hand-side of~\eqref{eq:avux} in terms of $(\tilde\bu,\tilde p,\tilde\str)$. 

We now give the detailed system of equations~(\ref{eq:divu}--\ref{eq:str0}) in the 2-d geometry of interest:
\begin{subequations}
\begin{alignat}{1}
\label{eq:divu0}
& \pds{x}u_x + \pds{z}u_z = 0,
\\
\label{eq:ux}
& \pds{t}u_x + u_x\pds{x}u_x + u_z\pds{z}u_x = -\pds{x}p + \pds{x}\tau_{xx} + \pds{z}\tau_{xz},
\\
\label{eq:uz}
& \pds{t}u_z + u_x\pds{x}u_z + u_z\pds{z}u_z = -\pds{z}p + \pds{x}\tau_{xz} + \pds{z}\tau_{zz} - g,
\\
\label{eq:txx}
& \pds{t}\tau_{xx} + u_x\pds{x}\tau_{xx} + u_z\pds{z}\tau_{xx} = (2\pds{x}u_x)\tau_{xx} + (2\pds{z}u_x)\tau_{xz} 
+ \frac{\eta_p}\lambda\pds{x}u_x
- \frac1\lambda\tau_{xx},
\\
\label{eq:tzz}
& \pds{t}\tau_{zz} + u_x\pds{x}\tau_{zz} + u_z\pds{z}\tau_{zz} = (2\pds{x}u_z)\tau_{xz} + (2\pds{z}u_z)\tau_{zz} 
+ \frac{\eta_p}\lambda\pds{z}u_z
-\frac1\lambda\tau_{zz},
\\
\label{eq:txz}
& \pds{t}\tau_{xz} + u_x\pds{x}\tau_{xz} + u_z\pds{z}\tau_{xz} = (\pds{x}u_z)\tau_{xx} + (\pds{z}u_x)\tau_{zz} 
+ \frac{\eta_p}{2\lambda}(\pds{z}u_x + \pds{x}u_z)
- \frac1\lambda\tau_{xz},
\end{alignat}
\end{subequations}
where we have used~(\ref{eq:divu0}) to simplify~(\ref{eq:txz}). The boundary conditions (\ref{eq:pureslip}), (\ref{eq:nofriction}) and (\ref{eq:notension}) write:
\begin{subequations}
\begin{alignat}{2}
\label{eq:tang}
&
\quad u_z= (\pds{x}b) u_x
&& \text{ at } z=b\,,
\\
\label{eq:nofriction2}
&
-(\pds{x}b)\tau_{xx}+\tau_{xz} 
= -\pds{x}b \Bigl(  -(\pds{x}b)\tau_{xz}+\tau_{zz} \Bigr)
&& \text{ at } z=b\,,
\\\label{eq:notension1}
&
-\pds{x}(b+h)(p-\tau_{xx})-\tau_{xz} = 0
&& \text{ at } z=b+h\,,
\\
\label{eq:notension2}
&
\quad\pds{x}(b+h)\tau_{xz}+(p-\tau_{zz}) = 0
&& \text{ at } z=b+h,
\end{alignat}
\end{subequations}
while the kinematic condition (\ref{eq:kinematic}), following (\ref{eq:h1}), writes
\begin{equation}
	\pds{t}h + \pds{x}\left(\int_b^{b+h}u_x\, dz\right)=0.
	\label{eq:kinem}
\end{equation}


We first simplify the derivation of a thin layer regime by assuming that the tangent of the angle between $\bn$ and $\be_z$ is uniformly small 
$$ \text{(H2) } \pds{x}b = O(\e) \text{ as } \e\to0 \,, $$ 
hence only smooth topographies with small slopes are treated here. This restriction could probably be alleviated following the ideas exposed in~\cite{bouchut-westdickenberg-2004}, though at the price of complications that seem unnecessary for a first presentation of our reduced model. On the contrary, the following assumptions are essential:
$$ \text{(H3) } \eta_p = O(\e),\qquad \lambda = O(1).$$
(As explained previously, we recall that the assumptions (H3) hold in the unit of the variable, which is here $L^2/T$ and $T$ respectively.)
As usual in Saint Venant models for avalanche flows, we are looking for solutions without small scale in $t$ and $x$ (thus only with scales $T$ and $L$), but with scale of order $\e$ in $z$ (in fact, $\e L$), which can be written formally as
\begin{equation}
	\pds{t}=O(1),\quad \pds{x}=O(1),\quad\pds{z}=O(1/\e)
	\label{eq:scalvar}
\end{equation}
in the respective units $1/T,1/L,1/L$. From now on, for the sake of simplicity, we shall not write explicitly the units as functions of $T$ and $L$ wherever they come into play.

We are looking for solutions with bounded velocity $\bu$ with bounded gradient $\nabla\bu$. Thus according to (\ref{eq:scalvar}) and to (\ref{eq:tang}), we are led to the
following assumptions on the orders of magnitude 
$$
\text{(H4) }
u_x = O(1), \quad u_z = O(\e), \quad \pds{z}u_x  = O(1), \quad \text{ as } \e\to0.
$$
(A typical profile for $u_x$ reads $A(t/T,x/L) + z B(t/T,x/L)$, with any dimensional functions $A$ and $B$ of the adimensional variables $t/T$ and $x/L$.)

According to (\ref{eq:ucm}), a typical value for $\str$ is $\eta_p\Dbu$. Thus we assume accordingly that
$$
\text{(H5) }
\str = O(\e) 
\quad \text{ as } \e\to0.
$$
We deduce from above that there exists some function $u_x^0(t,x)$ depending only on $(t,x)$ such that
\begin{equation}
	u_x(t,x,z)= u_x^0(t,x) + O(\e).
	\label{eq:expu1}
\end{equation}
Then, following the classical procedure \cite{gerbeau-perthame-2001,bouchut-mangeneycastelnau-perthame-vilotte-2003,bouchut-westdickenberg-2004,marche-2007}, we find the following successive implications.
\begin{enumerate}[i{)}]
 \item From the equation~(\ref{eq:uz}) on the vertical velocity $u_z$, we get by neglecting terms in $O(\e)$
\begin{equation}
\pds{z}p = \pds{z}\tau_{zz} - g + O(\e) \,.
\end{equation}
Hence $\pds{z}p=O(1)$, and the boundary condition~(\ref{eq:notension2}) gives that $ p = O(\e) $, indeed
\begin{equation}\label{eq:p1}
p = \tau_{zz} + g(b + h - z) + O(\e^2) \,.
\end{equation}
 \item
Next, from the equation~(\ref{eq:ux}) on the horizontal velocity $u_x$ we get
\begin{equation}
\label{eq:ux1}
\pds{t}u_x^0 + u_x^0\pds{x}u_x^0 = \pds{z}\tau_{xz} + O(\e) \,.
\end{equation}
The boundary condition~(\ref{eq:nofriction2}) gives $\tau_{xz}|_{z=b} = O(\e^2)$, thus with (\ref{eq:ux1}) it yields
\begin{equation}
\label{eq:txz1}
\tau_{xz} = (\pds{t}u_x^0 + u_x^0\pds{x}u_x^0)(z-b) + O(\e^2) \,.
\end{equation}
In addition the boundary condition~(\ref{eq:notension1}) implies that $ \tau_{xz}|_{z=b+h} = O(\e^2) $. We conclude therefore that
\begin{equation}
	\pds{t}u_x^0 + u_x^0\pds{x}u_x^0 = O(\e),
	\qquad\tau_{xz} = O(\e^2).
	\label{eq:adveps}
\end{equation}
 \item
The previous result combined with the equation~(\ref{eq:txz}) on $\tau_{xz}$ implies $ \pds{z}u_x = O(\e) $, hence
\begin{equation}
\label{eq:uxscaling}
u_x(t,x,z)= u_x^0(t,x) + O(\e^2) \,.
\end{equation}
This ``motion by slices'' property is stronger than the original one (\ref{eq:expu1}).
 \item
Using~(\ref{eq:uxscaling}) and (\ref{eq:p1}) in~(\ref{eq:ux}) improves~(\ref{eq:ux1}) to
\begin{equation}
\label{eq:ux2}
\pds{t}u_x^0 + u_x^0\pds{x}u_x^0 = \pds{x}(\tau_{xx}-\tau_{zz}-g(b+h)) + \pds{z}\tau_{xz} + O(\e^2) \,,
\end{equation}
which gives, with the boundary condition~(\ref{eq:nofriction2})
$\left[\tau_{xz}-\pds{x}b(\tau_{xx}-\tau_{zz})\right]|_{z=b} = O(\e^3)$,
\begin{equation}\begin{array}{l}
	\displaystyle \tau_{xz} 
	= \left[\pds{x}b(\tau_{xx}-\tau_{zz})\right]|_{z=b} - \int_b^z \pds{x}(\tau_{xx}-\tau_{zz})\, dz\\
	\displaystyle\hphantom{\tau_{xz} =}
	 + \left(\pds{t}u_x^0 + u_x^0\pds{x}u_x^0 + g\pds{x}(b+h)\right)(z-b) + O(\e^3).
	\label{eq:tauxzb}
	\end{array}
\end{equation}
But according to~(\ref{eq:notension1}) combined with~(\ref{eq:p1}), one has
$\left[\tau_{xz}-\pds{x}(b+h)(\tau_{xx}-\tau_{zz})\right]|_{z=b+h} = O(\e^3)$,
thus with (\ref{eq:ux2})
\begin{equation}\begin{array}{l}
	\displaystyle \tau_{xz} 
	= \left[\pds{x}(b+h)(\tau_{xx}-\tau_{zz})\right]|_{z=b+h} - \int_{b+h}^z \pds{x}(\tau_{xx}-\tau_{zz})\, dz \\
	\displaystyle\hphantom{\tau_{xz} =}
	+ \left(\pds{t}u_x^0 + u_x^0\pds{x}u_x^0 + g\pds{x}(b+h)\right)(z-b-h) + O(\e^3).
	\label{eq:tauxzb+h}
	\end{array}
\end{equation}
Therefore, the difference of (\ref{eq:tauxzb}) and (\ref{eq:tauxzb+h}) yields
\begin{equation}
\label{eq:ux3}
\left(\pds{t}u_x^0 + u_x^0\pds{x}u_x^0 + g\pds{x}(b+h)\right)h = \pds{x} \left(\int_{b}^{b+h} (\tau_{xx}-\tau_{zz})\, dz \right) + O(\e^3) \,.
\end{equation}
We note that $\tau_{xz}$ is then given by (\ref{eq:tauxzb}) or (\ref{eq:tauxzb+h}) as a function of $u^0_x$ and $(\tau_{xx}-\tau_{zz})$, and the evolution equation~(\ref{eq:ux3}) for $u^0_x$ is exactly the one that one would have obtained after integrating~(\ref{eq:ux2}) in the $\be_z$ direction and using the boundary conditions~(\ref{eq:nofriction2}) and~(\ref{eq:notension1}) combined with~(\ref{eq:p1}). It can also be obtained from (\ref{eq:avux}).
 \item The result~(\ref{eq:uxscaling}) with the incompressibility condition~(\ref{eq:divu0}) and the impermeability condition~(\ref{eq:tang}) at the bottom also allows to compute the vertical component of the velocity
\begin{equation}
\label{eq:uz1} 
u_z = (\pds{x}b)u_x|_{z=b} - \int_b^z \pds{x}u_x\, dz  = (\pds{x}b)u^0_x - (z-b)\pds{x}u^0_x + O(\e^3) \,,
\end{equation}
which is of course consistent with our hypotheses about $u_z=O(\e)$.
 \item Collecting all the previous results,~(\ref{eq:txx}) and~(\ref{eq:tzz}) up to $ O(\e^2) $ give
\begin{equation}
\label{eq:tauxx-tauzz}
\left\lbrace
\begin{aligned}
& \pds{t}\tau_{xx} + u_x^0\pds{x}\tau_{xx} + \left((\pds{x}b)u^0_x - (z-b)\pds{x}u^0_x\right)\pds{z}\tau_{xx} = 2(\pds{x}u_x^0)\tau_{xx}
+ \frac{\eta_p\pds{x}u_x^0- \tau_{xx}}\lambda
+ O(\e^2),
\\
& \pds{t}\tau_{zz} + u_x^0\pds{x}\tau_{zz} + \left((\pds{x}b)u^0_x - (z-b)\pds{x}u^0_x\right)\pds{z}\tau_{zz} = - 2(\pds{x}u_x^0)\tau_{zz}
- \frac{\eta_p\pds{x}u_x^0+ \tau_{zz}}\lambda
+ O(\e^2),
\end{aligned}
\right.
\end{equation}
which closes the system of equations for the reduced model.
 \item The previous results which give $\tau_{xz}$ at order $ O(\e^3) $, that is~(\ref{eq:tauxzb}) or~(\ref{eq:tauxzb+h}), are consistent with the equation~(\ref{eq:txz}) for $\tau_{xz}$ at order $ O(\e^3) $, from which one could next obtain an approximation for $\pds{z}u_x$ up to $ O(\e^2) $, that is
  \begin{multline}
  \pds{t}\tau_{xz} + u_x^0\pds{x}\tau_{xz} + ((\pds{x}b) u_x^0  + (z-b) \pds{x} u_x^0)\pds{z}\tau_{xz} 
  + \frac1\lambda\tau_{xz}
  \\
  = \pds{x}\left((\pds{x}b) u_x^0  (z-b) + \pds{x} u_x^0\right)\left(\tau_{xx}+\frac{\eta_p}{2\lambda}\right)
  +  \pds{z}u_x \left(\tau_{zz} + \frac{\eta_p}{2\lambda}\right) 
  \end{multline}
  with $\tau_{xx}$ and $\tau_{zz}$ given up to order $ O(\e^2) $ by~(\ref{eq:tauxx-tauzz}). This procedure fixes the next term in the expansion (\ref{eq:uxscaling}).
  Note in particular that {\sl we do not have} $u_x(t,x,z)= u_x^0(t,x) + O(\e^3)$ (dependence on the vertical coordinate subsists at order $\e^2$).
\end{enumerate}

To sum up, dropping $\e$, we have obtained a closed system of equations
\begin{equation}\label{eq:reduced}
\left\lbrace
\begin{aligned}
& \pds{t}h + \pds{x}(h u_x^0) = 0,
\\
& \pds{t}(h u_x^0) + \pds{x}\left( h(u_x^0)^{2} + g\frac{h^2}2 + \int_{b}^{b+h} (\tau_{zz}-\tau_{xx})\, dz \right) = 
-g(\pds{x}b)h,
\\
& \pds{t}\tau_{xx} + u_x^0\pds{x}\tau_{xx} + \left((\pds{x}b)u^0_x - (z-b)\pds{x}u^0_x\right)\pds{z}\tau_{xx} = 2(\pds{x}u_x^0)\tau_{xx}
+ \frac{\eta_p}\lambda\pds{x}u_x^0
- \frac1\lambda\tau_{xx},
\\
& \pds{t}\tau_{zz} + u_x^0\pds{x}\tau_{zz} + \left((\pds{x}b)u^0_x - (z-b)\pds{x}u^0_x\right)\pds{z}\tau_{zz} = - 2(\pds{x}u_x^0)\tau_{zz}
- \frac{\eta_p}\lambda\pds{x}u_x^0
-\frac1\lambda\tau_{zz},
\end{aligned}
\right.
\end{equation}
which allows to compute consistently uniform asymptotic approximations of $(u_x,u_z,p,\tau_{xx},\tau_{zz},\tau_{xz})$ as variables of order $O(\e^{(0,1,1,1,1,2)})$,
up to errors in $O(\e^{(2,3,2,2,2,3)})$. These correspond to approximations of (\ref{eq:divu0})-(\ref{eq:kinem}) up to $O(\e^{(2,2,1,2,2,3,3,3,3,2,3)})$.

In (\ref{eq:reduced}), $b$ depends only on $x$, $h$ and $u_x^0$ depend on $(t,x)$, while $\tau_{xx}$ and $\tau_{zz}$ depend on $(t,x,z)$. 
However, observe that the momentum conservation equation invokes only $\int_b^{b+h}\tau_{xx}dz$ and $\int_b^{b+h}\tau_{zz}dz$, which do not depend on $z$.
Now, using Leibniz rule and the boundary conditions, it is possible to get equations for $\int_b^{b+h}\tau_{xx}dz$ and $\int_b^{b+h}\tau_{zz}dz$ (integrating those for $\tau_{xx}$ and $\tau_{zz}$, see~\eqref{eq:avadvphi} below in Section~\ref{sec:properties}) and form a closed system with the equations for the momentum and mass conservation.
Another equivalent way to derive the same closed system of equations is to assume  that $\tau_{xx}$ and $\tau_{zz}$ are independent of $z$ (at least at first-order in $\e$).
In the rest of this paper, we shall mainly be concerned with that simplified system of equations, whose mathematical properties are easier to study.

\section{The new reduced model and its mathematical properties}
\label{sec:properties}


The reduced model (\ref{eq:reduced}) is endowed with an energy equation similar to the one for the full UCM model. Obviously, the whole system of equations for $\str$ in the reduced model rewrite with the entries of the conformation tensor $\strs= \I + \frac{2\lambda}{\eta_p}\str$. However, since it is diagonal at leading order, we consider only the diagonal part
\begin{equation}
\label{eq:sigma0}
\strs^0 = 
\begin{pmatrix}\sigma_{xx}=1+\frac{2\lambda}{\eta_p}\tau_{xx}&0\\0&\sigma_{zz}=1+\frac{2\lambda}{\eta_p}\tau_{zz}\end{pmatrix}.
\end{equation}
The two last equations of (\ref{eq:reduced}) yield
\begin{equation}
\left\lbrace
\begin{aligned}
& \pds{t}\sigma_{xx} + u_x^0\pds{x}\sigma_{xx} + \left((\pds{x}b)u^0_x - (z-b)\pds{x}u^0_x\right)\pds{z}\sigma_{xx} = 2(\pds{x}u_x^0)\sigma_{xx}
- \frac1\lambda(\sigma_{xx}-1),
\\
& \pds{t}\sigma_{zz} + u_x^0\pds{x}\sigma_{zz} + \left((\pds{x}b)u^0_x - (z-b)\pds{x}u^0_x\right)\pds{z}\sigma_{zz} = - 2(\pds{x}u_x^0)\sigma_{zz}
-\frac1\lambda(\sigma_{zz}-1).
\end{aligned}
\right.
\label{eq:strs1}
\end{equation}
These equations imply that $\sigma_{xx}$ and $\sigma_{zz}$ remain positive if they are initially.
Then, we compute
\begin{equation}
	\label{eq:logsigma}
\begin{array}{l}
\displaystyle \Bigl(\pds{t} + u_x^0\pds{x} + \left((\pds{x}b)u^0_x - (z-b)\pds{x}u^0_x\right)\pds{z}\Bigr)\bigl(\frac{1}{2}\tau_{xx}-\frac{\eta_p}{4\lambda}\ln\bigl(1+\frac{2\lambda}{\eta_p}\tau_{xx}\bigr)\bigr)
	=(\pds{x}u^0_x)\tau_{xx}-\frac{1}{\eta_p}\frac{\tau_{xx}^2}{\sigma_{xx}},\\
	\displaystyle\Bigl(\pds{t} + u_x^0\pds{x} + \left((\pds{x}b)u^0_x - (z-b)\pds{x}u^0_x\right)\pds{z}\Bigr)\bigl(\frac{1}{2}\tau_{zz}-\frac{\eta_p}{4\lambda}\ln\bigl(1+\frac{2\lambda}{\eta_p}\tau_{zz}\bigr)\bigr)
	=-(\pds{x}u^0_x)\tau_{zz}-\frac{1}{\eta_p}\frac{\tau_{zz}^2}{\sigma_{zz}}.
\end{array}
\end{equation}
In order to compute the integral of (\ref{eq:logsigma}) with respect to $z$, we notice the following formula for any function $\varphi(t,x,z)$ (a combination of the Leibniz rule with boundary conditions at $z=b$ and $z=b+h$), 
\begin{equation}
	\label{eq:avadvphi}
\begin{array}{l}
\displaystyle\hphantom{=\,} \int_b^{b+h}\Bigl(\pds{t} + u_x^0\pds{x} + \left((\pds{x}b)u^0_x - (z-b)\pds{x}u^0_x\right)\pds{z}\Bigr)\varphi\,dz\\
	\displaystyle =\int_b^{b+h}\biggl(\pds{t}\varphi + \pds{x}(u_x^0\varphi) +\pds{z}\Bigl( \left((\pds{x}b)u^0_x - (z-b)\pds{x}u^0_x\right)\varphi\Bigr)\biggr)\,dz\\
	\displaystyle =\pds{t}\int_b^{b+h}\mkern -11mu\varphi\,dz-\varphi_{b+h}\pds{t}h
	+\pds{x}\int_b^{b+h}\mkern -11mu u_x^0\varphi\,dz-(u_x^0\varphi)_{b+h}\pds{x}(b+h)
	+(u_x^0\varphi)_b\partial_x b\\
	\displaystyle\mkern 80mu+\Bigl((\pds{x}b)u^0_x - h\pds{x}u^0_x\Bigr)\varphi_{b+h}-(\pds{x}b)u^0_x\varphi_b\\
	\displaystyle =\pds{t}\int_b^{b+h}\mkern -11mu\varphi\,dz
	+\pds{x}\biggl(u_x^0\int_b^{b+h}\mkern -11mu \varphi\,dz\biggr).
\end{array}
\end{equation}
Therefore, summing up the two equations of (\ref{eq:logsigma}) and integrating in $z$ gives
\begin{equation}
	\label{eq:ensigma}
\begin{array}{l}
\displaystyle\hphantom{=\ } \pds{t}\int_b^{b+h}\frac{\eta_p}{4\lambda}\tr(\strs^0-\ln\strs^0-\I)\,dz
	+\pds{x}\biggl(u_x^0\int_b^{b+h}\frac{\eta_p}{4\lambda}\tr(\strs^0-\ln\strs^0-\I)\,dz\biggr)\\
	\displaystyle =(\pds{x}u^0_x)\int_b^{b+h}\bigl(\tau_{xx}-\tau_{zz}\bigr)\,dz
	-\frac{1}{\eta_p}\int_b^{b+h}\biggl(\frac{\tau_{xx}^2}{\sigma_{xx}}+\frac{\tau_{zz}^2}{\sigma_{zz}}\biggr)\,dz.
\end{array}
\end{equation}
Moreover, the classical computation of energy for the Saint Venant model gives
\begin{equation}
\label{eq:ensw}
	\pds{t}\Bigl(h\frac{(u_x^0)^2}{2}+g\frac{h^2}{2}+gbh\Bigr)
	+\pds{x}\Bigl(\bigl(h\frac{(u_x^0)^2}{2}+gh^2+gbh\bigr)u_x^0\Bigr)
	+u_x^0\,\pds{x}\int_b^{b+h}(\tau_{zz}-\tau_{xx})\,dz=0.
\end{equation}
Adding up~(\ref{eq:ensw}) and~(\ref{eq:ensigma}) yields
\begin{equation}
	\label{eq:estimate1}
\begin{array}{l}
\displaystyle \pds{t}\left( h \frac{(u_x^0)^2}{2} + g\frac{h^2}{2} + gbh
+ \frac{\eta_p}{4\lambda} \int_b^{b+h}\tr(\strs^0-\ln\strs^0-\I)\,dz \right)
\\
\displaystyle+ \pds{x}\left(\left( h\frac{(u_x^0)^2}{2} + gh^2 + gbh
+ \frac{\eta_p}{4\lambda}\int_b^{b+h}\tr(\strs^0-\ln\strs^0-\I)\,dz 
+ \frac{\eta_p}{2\lambda}\int_{b}^{b+h} (\sigma_{zz}-\sigma_{xx})\,dz\right) u_x^0\right)
\\
\displaystyle=- \frac{\eta_p}{4\lambda^2} \int_b^{b+h}\tr(\strs^0+[\strs^0]^{-1}-2\I)\,dz.
\end{array}
\end{equation}
Therefore, we get an exact energy identity for solutions to the reduced model (\ref{eq:reduced}). Note that to discriminate between possibly many discontinuous solutions 
(generalized solutions in a sense to be defined, see below the discussion on the conservative formulation), we would naturally require an inequality in (\ref{eq:estimate1}) instead of an equality.

In the case of $\tau_{xx}$ and $\tau_{zz}$ independent of $z$, everything becomes more explicit. Using the variables $\sigma_{xx}=1+\frac{2\lambda}{\eta_p}\tau_{xx}$ and $\sigma_{zz}=1+\frac{2\lambda}{\eta_p}\tau_{zz}$ (also clearly independent of $z$), the simplified reduced model then writes
\begin{equation}
\label{eq:reducedmodel} 
\boxed{
\left\lbrace
\begin{aligned}
& \pds{t}h + \pds{x}(h u_x^0) = 0,
\\
& \pds{t}(h u_x^0) + \pds{x}\left( h(u_x^0)^{2} + g\frac{h^2}2 + \frac{\eta_p}{2\lambda}h(\sigma_{zz}-\sigma_{xx}) \right) = 
-gh\pds{x}b,
\\
& \pds{t}\sigma_{xx} + u_x^0\pds{x}\sigma_{xx}-2\sigma_{xx}\pds{x}u_x^0
=  \frac{1-\sigma_{xx}}\lambda,
\\
& \pds{t}\sigma_{zz} + u_x^0\pds{x}\sigma_{zz}+2\sigma_{zz}\pds{x}u_x^0
=  \frac{1-\sigma_{zz}}\lambda,
\end{aligned}
\right.
}
\end{equation}
while the energy inequality becomes ($\strs^0$ is defined in (\ref{eq:sigma0}))
\begin{equation}
\label{eq:estimate2}
\begin{array}{l}
\displaystyle \pds{t}\left( h \frac{(u_x^0)^2}{2} + g\frac{h^2}{2} + gbh
+ \frac{\eta_p}{4\lambda} h\tr(\strs^0-\ln\strs^0-\I) \right)\\
\displaystyle+ \pds{x}\left(\left( h\frac{(u_x^0)^2}{2} + gh^2 + gbh
+ \frac{\eta_p}{4\lambda}h\tr(\strs^0-\ln\strs^0-\I) 
 + \frac{\eta_p}{2\lambda} h (\sigma_{zz}-\sigma_{xx})\right) u_x^0\right)
\\
\displaystyle \leq - \frac{\eta_p}{4\lambda^2} h\tr(\strs^0+[\strs^0]^{-1}-2\I).
\end{array}
\end{equation}
In (\ref{eq:reducedmodel}) and (\ref{eq:estimate2}), $b$ is a function of $x$ and $h$, $u_x^0$, $\sigma_{xx}$, $\sigma_{zz}$ depend on $(t,x)$, with $h\geq 0$, $\sigma_{xx}\geq 0$, $\sigma_{zz}\geq 0$. {\it From now on, we shall only deal with the simplified reduced model}~(\ref{eq:reducedmodel}).

The inequality~(\ref{eq:estimate2}) (instead of equality) for possibly discontinuous solutions rules out generalized solutions for which the dissipation -- already present in our model ! -- is physically {\sl not enough} (see also~\cite{boyaval-lelievre-mangoubi-2009} where a similar numerical ``entropy'' condition is used to build stable finite-element schemes for the viscous UCM model, namely the so-called Oldroyd-B model).

\begin{remark}[Limit cases]\label{rem:limit}
For the system (\ref{eq:reducedmodel}), two interesting regimes are important to mention. 
The first is the standard Saint Venant regime, for which one takes $\eta_p/\lambda=0$.
It is obtained in the limit $\eta_p\to0$ for fixed $\lambda$
(as opposed to the limit $\lambda\to\infty$ for fixed $\eta_p$, some kind of ``High-Weissenberg limit''~\cite{renardy-2000} which is problematic,
as we will see in the numerical experiments).
The second regime is obtained in the ``Low-Weissenberg limit'' $\lambda\rightarrow 0$, for fixed $\eta_p$. 
Assuming $(1-\sigma_{xx})/\lambda$ and $(1-\sigma_{zz})/\lambda$ remain bounded, 
the system rewritten with $\tau_{xx}$ and $\tau_{zz}$ gives the viscous Saint Venant system
\begin{equation}\label{eq:viscsv}
\left\lbrace
\begin{aligned}
& \pds{t}h + \pds{x}(h u_x^0) = 0,
\\
& \pds{t}(h u_x^0) + \pds{x}\left( h(u_x^0)^{2} + g\frac{h^2}2 -2\eta_p h\,\pds{x}u_x^0\right) = 
-gh\pds{x}b,
\end{aligned}
\right.
\end{equation}
with the energy inequality
\begin{equation}
\label{eq:viscousenergy}
\pds{t}\left( h \frac{(u_x^0)^2}{2} + g\frac{h^2}{2} + gbh
\right)+ \pds{x}\left(\left( h\frac{(u_x^0)^2}{2} + gh^2 + gbh 
 -2\eta_p h\,\pds{x}u_x^0\right) u_x^0\right)
\leq- 2\eta_p h(\pds{x}u_x^0)^2
\,.
\end{equation}
\end{remark}

\begin{remark}[Steady states]\label{rem:steady}
The source terms $(1-\sigma_{xx})/\lambda$ and $(1-\sigma_{zz})/\lambda$ in (\ref{eq:reducedmodel}) are responsible for the right-hand side that dissipates energy in (\ref{eq:estimate2}). This dissipation has the consequence that steady states are possible only if
\begin{equation}
\tr(\strs^0+[\strs^0]^{-1}-2\I) = 0,\quad\mbox{i.e. }\str=0 \,,
\end{equation}
which implies that steady solutions to~(\ref{eq:reducedmodel}) identify with the steady solutions at rest to the standard Saint Venant model: $u^0_x = 0$, $h+b = cst$, $\sigma_{xx}=\sigma_{zz}=1$.
\end{remark}

\begin{remark}[Conservativity]\label{rem:conservativity}
The reduced model~(\ref{eq:reducedmodel}) is a first-order quasilinear system with source, but not written in conservative form because of the stress equations on $\sigma_{xx}$ and $\sigma_{zz}$. Indeed, one can put them in conservative form as follows,
\begin{equation}
\label{eq:constau}
\left\{
\begin{array}{ll}
\displaystyle
\pds{t}\left((\sigma_{xx})^{-1/2}\right) +\pds{x}\left((\sigma_{xx})^{-1/2}u_x^0\right)
=
&
\displaystyle
-\sigma_{xx}^{-3/2}\frac{1-\sigma_{xx}}{2\lambda},\vphantom{\Biggl|}
\\
\displaystyle
\pds{t}\left((\sigma_{zz})^{1/2}\right) +\pds{x}\left((\sigma_{zz})^{1/2}u_x^0\right)
=
&
\displaystyle
\sigma_{zz}^{-1/2}\frac{1-\sigma_{zz}}{2\lambda}.
\end{array}
\right.
\end{equation}
However, these conservative equations do not help since they are physically irrelevant.
Moreover, the physical energy of~(\ref{eq:estimate2}) is not convex with respect to these conservative variables $\sigma_{xx}^{-1/2}$ 
and $\sigma_{zz}^{1/2}$. 
As a matter of fact, one can show that the energy, that is 
\begin{align}
\label{eq:energy}
\widetilde E
&= h \frac{(u_x^0)^2}{2} + g\frac{h^2}{2} + gbh + \frac{\eta_p}{4\lambda}
h \left(\sigma_{xx}+\sigma_{zz}-\ln(\sigma_{xx}\sigma_{zz})-2\right)
\,,
\end{align}
cannot be convex with respect to any set of conservative variables of the form
\begin{equation}
\label{eq:generalconservativevariables}
\left(h,h u^0_x,h\varpi^{-1}\left(\frac{\sigma_{xx}^{-1/2}}{h}\right),h\varsigma^{-1}\left(\frac{\sigma_{zz}^{1/2}}{h}\right)\right),
\end{equation}
where $\varpi,\varsigma$ are smooth functions standing for general changes of variables, see Appendix~\ref{app:convexity}.
\end{remark}
Nevertheless, the system (\ref{eq:reducedmodel}) can be written in the following canonical form, strongly reminiscent of the gas dynamics system,
\begin{equation}
\label{eq:canonical}
\left\lbrace
\begin{aligned}
& \pds{t}h + \pds{x}(h u_x^0) = 0,
\\
& \pds{t}(h u_x^0) + \pds{x}\left( h(u_x^0)^{2} + P(h,\bs) \right) = 
-gh\pds{x}b,
\\
& \pds{t}\bs + u_x^0\pds{x}\bs
=  \frac{1}\lambda\cs(h,\bs),
\end{aligned}
\right.
\end{equation}
with
\begin{equation}
	\bs=\bigl(s_{xx},s_{zz}\bigr)=\left(\frac{\sigma_{xx}^{-1/2}}{h},\frac{\sigma_{zz}^{1/2}}{h}\right),
	\label{eq:bs}
\end{equation}
\begin{equation}
	\cs(h,\bs)=\left(-\frac{\sigma_{xx}^{-3/2}}{2h}(1-\sigma_{xx})
	,\frac{\sigma_{zz}^{-1/2}}{2h}(1-\sigma_{zz})\right),
	\label{eq:cs}
\end{equation}
\begin{equation}
	P(h,\bs)=g\frac{h^2}2 + \frac{\eta_p}{2\lambda}h(\sigma_{zz}-\sigma_{xx}).
	\label{eq:P}
\end{equation}
One can compute
\begin{equation}
	\left(\frac{\partial P}{\partial h}\right)_{|{\bs}}=gh+\frac{\eta_p}{2\lambda}(\sigma_{zz}-\sigma_{xx}+h\frac{2\sigma_{zz}}{h}+h\frac{2\sigma_{xx}}{h})
	=gh+\frac{\eta_p}{2\lambda}(3\sigma_{zz}+\sigma_{xx})>0,
	\label{eq:dPdh}
\end{equation}
from which we conclude that for smooth $b$, the system \eqref{eq:canonical} is hyperbolic with eigenvalues
\begin{equation}
\label{eq:eigenvalues}
\lambda_1 = u^0_x - \sqrt{gh+\frac{\eta_p}{2\lambda}(3\sigma_{zz}+\sigma_{xx})}, 
\qquad
\lambda_2= u^0_x,
\qquad
\lambda_3 = u^0_x + \sqrt{gh+\frac{\eta_p}{2\lambda}(3\sigma_{zz}+\sigma_{xx})}, 
\end{equation}
the second having double multiplicity. One can check that $\lambda_2$ is linearly degenerate, while $\lambda_1$ and $\lambda_3$ are genuinely nonlinear (this follows from 
computations similar to~\cite[Example 2.4 p.45]{godlewski-raviart-1996} and the first line of \eqref{eq:proppress}).

From the particular formulation \eqref{eq:canonical}, one sees that the jump conditions for a $2-$contact discontinuity are that $u^0_x$ and $P$ do not jump (as weak 2-Riemann invariants). However, jump conditions across $1-$ and $3-$shocks need to be chosen in order to determine weak discontinuous solutions in a unique way.

A possible choice of jump conditions is, as explained in Remark \ref{rem:conservativity}, to take the conservative formulation \eqref{eq:constau} (or equivalently a conservative
formulation related to the variables \eqref{eq:generalconservativevariables}, leading to the condition that $\bs$ does not jump through $1-$ and $3-$shocks). This formulation gives unphysical conservations and nonconvex energy (which could produce numerical under/overshoots), and we shall not make this choice.

Our choice of jump conditions will be rather imposed indirectly by numerical considerations, via the choice of a set of pseudo-conservative variables, i.e.
variables for which we shall write discrete flux difference equations. Solving nonconservative systems leads in general to convergence to unexpected solutions, as
explained in \cite{castro-et-al-2008}. With a pragmatical point of view, we nevertheless choose the pseudo-conservative variables as
\begin{equation}
\label{eq:variables}
q \equiv (q_1,q_2,q_3,q_4)^T := \left(h,h u^0_x,h\sigma_{xx},h\sigma_{zz}\right)^T.
\end{equation}
In other words, we consider the formal system
\begin{equation}
\label{eq:reducedmodelpseudocons} 
\left\lbrace
\begin{aligned}
& \pds{t}h + \pds{x}(h u_x^0) = 0,
\\
& \pds{t}(h u_x^0) + \pds{x}\left( h(u_x^0)^{2} + g\frac{h^2}2 + \frac{\eta_p}{2\lambda}h(\sigma_{zz}-\sigma_{xx}) \right) = 
-gh\pds{x}b,
\\
& \pds{t}(h\sigma_{xx}) + \pds{x}(h\sigma_{xx}u_x^0)-2h\sigma_{xx}\pds{x}u_x^0
=  \frac{h-h\sigma_{xx}}\lambda,
\\
& \pds{t}(h\sigma_{zz}) + \pds{x}(h\sigma_{zz}u_x^0)+2h\sigma_{zz}\pds{x}u_x^0
=  \frac{h-h\sigma_{zz}}\lambda.
\end{aligned}
\right.
\end{equation}
The choice of these pseudo-conservative variables is good for at least two reasons:
\begin{itemize}
 \item these variables are physically relevant,
 \item the energy $\widetilde E$ in \eqref{eq:energy} is convex with respect to them (see Appendix~\ref{app:convexity}).
\end{itemize}
The second point will make it easier to build a discrete scheme that is energy satisfying (in the sense of the energy inequality~\eqref{eq:estimate2}), while preserving the convex (in the variable $q$) set
\begin{equation}
\label{eq:invariantdomain}
\mathcal{U} = \{h\ge0,\sigma_{xx}\ge0,\sigma_{zz}\ge0\} \,,
\end{equation}
which is here the physical invariant domain where the energy inequality~\eqref{eq:estimate2} makes sense. Note that our system is of the form considered in \cite{berthon-coquel-lefloch-2011} (see also Remark \ref{Remark maxprinc}).

Let us mention that for the viscous UCM model, namely the Oldroyd-B model, various numerical techniques are proposed in \cite{lozinski-owens-2003,lee-xu-2006,boyaval-lelievre-mangoubi-2009,barrett-boyaval-2011} for the preservation of the positive-definiteness of a non-necessarily diagonal tensor $\strs$ in the context of finite-element discretizations.

\section{Finite volume method and numerical results}
\label{sec:numerics}

In this section we describe a finite volume approximation of ~\eqref{eq:reducedmodelpseudocons}.
The approximation of the full system is achieved by a fractional step approach, discretizing successively the system~\eqref{eq:reducedmodelpseudocons} without 
the relaxation source terms in $1/\lambda$ on the right-hand side of the two stress equations, and these relaxation terms alone.
The topographic source term $h\partial_x b$ is treated by the hydrostatic reconstruction method of \cite{audusse-bouchut-bristeau-klein-perthame-2004}
in Subsection \ref{sec:topo}.
This approach ensures that the whole scheme is {\sl well-balanced} with respect to the steady states of Remark \ref{rem:steady},
because the relaxation terms vanish for these solutions.

The integration of relaxation source terms is performed by a time-implicit cell-centered formula.
Note that then the scheme is not {\sl asymptotic preserving} with respect to the
viscous Saint Venant asymptotic regime $\lambda\rightarrow 0$ of Remark~\ref{rem:limit},
for this one would need a more complex treatment of these relaxation terms.

Let us now concentrate on the resolution of the system ~\eqref{eq:reducedmodelpseudocons}
without any source, i.e.
\begin{equation}
\label{eq:pseudocons} 
\left\lbrace
\begin{aligned}
& \pds{t}h + \pds{x}(h u_x^0) = 0,
\\
& \pds{t}(h u_x^0) + \pds{x}\left( h(u_x^0)^{2} + P \right) = 0,
\\
& \pds{t}(h\sigma_{xx}) + \pds{x}(h\sigma_{xx}u_x^0)-2h\sigma_{xx}\pds{x}u_x^0 =  0,
\\
& \pds{t}(h\sigma_{zz}) + \pds{x}(h\sigma_{zz}u_x^0)+2h\sigma_{zz}\pds{x}u_x^0 =  0,
\end{aligned}
\right.
\end{equation}
with
\begin{equation}
	P=g\frac{h^2}2 + \frac{\eta_p}{2\lambda}h(\sigma_{zz}-\sigma_{xx}),
	\label{eq:Ppseudo}
\end{equation}
and the energy inequality
\begin{equation}
\label{eq:entropypseudo}
\begin{array}{l}
\displaystyle \pds{t}\left( h\frac{(u_x^0)^2}{2} + g\frac{h^2}{2}
+ \frac{\eta_p}{4\lambda}h\bigl(\sigma_{xx}+\sigma_{zz}-\ln(\sigma_{xx}\sigma_{zz})-2\bigr) \right)\\
\displaystyle+ \pds{x}\left(\Bigl( h\frac{(u_x^0)^2}{2} + g\frac{h^2}{2}
+ \frac{\eta_p}{4\lambda}h\bigl(\sigma_{xx}+\sigma_{zz}-\ln(\sigma_{xx}\sigma_{zz})-2\bigr)  + P\Bigr) u_x^0\right)\leq 0.
\end{array}
\end{equation}
A finite volume scheme for the quasilinear system~\eqref{eq:pseudocons}-\eqref{eq:Ppseudo}
can be classically built following Godunov's approach, considering piecewise constant approximations
of $q=(h,hu^0_x,h\sigma_{xx},h\sigma_{zz})$, and invoking an approximate Riemann solver
at the interface between two cells.

\subsection{Approximate Riemann solver}
\label{sec:conservativescheme}

In order to get an approximate Riemann solver for \eqref{eq:pseudocons},
we use the standard relaxation approach, as described in \cite{bouchut-2004}.
It naturally handles the energy inequality~\eqref{eq:entropypseudo},
and also preserves the invariant domain~\eqref{eq:invariantdomain}.

Because of the canonical form of \eqref{eq:pseudocons}, which is \eqref{eq:canonical} without source, i.e.
\begin{equation}
\label{eq:canonicalhom}
\left\lbrace
\begin{aligned}
& \pds{t}h + \pds{x}(h u_x^0) = 0,
\\
& \pds{t}(h u_x^0) + \pds{x}\left( h(u_x^0)^{2} + P \right) = 0,
\\
& \pds{t}\bs + u_x^0\pds{x}\bs =  0,
\end{aligned}
\right.
\end{equation}
with
\begin{equation}
	\bs=\bigl(s_{xx},s_{zz}\bigr)=\left(\frac{\sigma_{xx}^{-1/2}}{h},\frac{\sigma_{zz}^{1/2}}{h}\right),
	\label{eq:bshom}
\end{equation}
we have a formal analogy with the system of full gas dynamics equations.
Therefore, we follow the usual Suliciu relaxation approach that is described in \cite{bouchut-2004}.
We introduce a new variable $\pi$, the relaxed pressure, and a variable $c>0$ intended to
parametrize the speeds. Then we solve the system
\begin{equation}
\label{eq:relax}
\left\{
\begin{aligned}
& \pds{t}h + \pds{x}(hu_x^0) = 0,
\\
& \pds{t}(hu_x^0) + \pds{x}(h(u_x^0)^2+\pi) = 0,
\\
& \pds{t}(h\pi/c^2) + \pds{x}(h\pi u_x^0/c^2 + u_x^0)= 0,
\\
& \pds{t}c + u_x^0\pds{x}c = 0,
\\
& \pds{t}\bs + u_x^0\pds{x}\bs =  0.
\end{aligned}
\right.
\end{equation}
This quasilinear system has the property of having a quasi diagonal form
\begin{equation}
\label{diagonalrelax}
\left\{
\begin{aligned}
&\pds{t}(\pi+cu^0_x) + (u^0_x+c/h)\pds{x}(\pi+cu^0_x)-\frac{u^0_x}{h}c\pds{x}c = 0 \,,
\\
&\pds{t}(\pi-cu^0_x) + (u^0_x-c/h)\pds{x}(\pi-cu^0_x)-\frac{u^0_x}{h}c\pds{x}c = 0 \,,
\\
&\pds{t}\left(1/h+\pi/c^2\right) + u^0_x\pds{x}\left(1/h+\pi/c^2\right) = 0 \,, 
\\
&\pds{t}c + u_x^0\pds{x}c =  0 \,,
\\
&\pds{t}\bs + u_x^0\pds{x}\bs =  0 \,.\end{aligned}
\right.
\end{equation}
One deduces its eigenvalues, which are $u^0_x-c/h$, $u^0_x+c/h$, and $u^0_x$ with multiplicity $4$.
One checks easily that the system is hyperbolic, with all eigenvalues linearly degenerate.
As a consequence, Rankine-Hugoniot conditions are well-defined (the weak Riemann invariants
do not jump through the associated discontinuity), and are equivalent to any conservative formulation.
We notice that with the relation \eqref{eq:bshom} the equation on $\bs$ in \eqref{eq:relax}
can be transformed back to
\begin{equation}\begin{array}{l}
	\dsp \pds{t}(h\sigma_{xx}) + \pds{x}(h\sigma_{xx}u_x^0)-2h\sigma_{xx}\pds{x}u_x^0 =  0,\\
	\dsp \pds{t}(h\sigma_{zz}) + \pds{x}(h\sigma_{zz}u_x^0)+2h\sigma_{zz}\pds{x}u_x^0 =  0.
	\label{eq:strrelax}
	\end{array}
\end{equation}
The approximate Riemann solver can be defined as follows, starting from left and right
values of $h,hu_x^0,h\sigma_{xx},h\sigma_{zz}$ at an interface :
\begin{itemize}
\item Solve the Riemann problem for \eqref{eq:relax} with initial data completed by
the relations
\begin{equation}
	\pi_l=P(h_l,(\sigma_{xx})_l,(\sigma_{zz})_l),
	\qquad\pi_r=P(h_r,(\sigma_{xx})_r,(\sigma_{zz})_r),
	\label{eq:pilpir}
\end{equation}
and with suitable values of $c_l$ and $c_r$ that will be discussed below.
\item Retain in the solution only the variables $h,hu_x^0,h\sigma_{xx},h\sigma_{zz}$.
The result is a vector called $R(x/t,q_l,q_r)$.
\end{itemize}
Note that this approximate Riemann solver $R(x/t,q_l,q_r)$ has the property to give the exact solution
for an isolated contact discontinuity (i.e. when the initial data is such that $u_x^0$
and $P$ are constant), because in this case the solution to \eqref{eq:relax} is
the solution to \eqref{eq:pseudocons} completed with $\pi=P(h,\bs)$.\\

Then, the numerical scheme is defined as follows. We consider a mesh of cells
$(x_{i-1/2},x_{i+1/2})$, $i\in\Z$, of length $\Delta x_i=x_{i+1/2}-x_{i-1/2}$,
discrete times $t_n$ with $t_{n+1}-t_n=\Delta t$, and
cell values $q_i^n$ approximating the average of $q$ over the cell $i$ at time $t_n$.
We can then define an approximate solution $q^{appr}(t,x)$ for $t_n\leq t<t_{n+1}$ and $x\in\R$ by
\begin{equation}
	q^{appr}(t,x)=R\left(\frac{x-x_{i+1/2}}{t-t_n},q_i^n,q_{i+1}^n\right)
	\quad\mbox{for }x_i<x<x_{i+1},
	\label{eq:approxsol}
\end{equation}
where $x_i=(x_{i-1/2}+x_{i+1/2})/2$. This definition is coherent under a half CFL condition,
formulated as
\begin{equation}\begin{array}{l}
	\dsp x/t<-\frac{\Delta x_i}{2\Delta t}\Rightarrow R(x/t,q_i,q_{i+1})=q_i,\\
	\dsp x/t>\frac{\Delta x_{i+1}}{2\Delta t}\Rightarrow R(x/t,q_i,q_{i+1})=q_{i+1}.
	\label{eq:CFLhalfsolver}
	\end{array}
\end{equation}
The new values at time $t_{n+1}$ are finally defined by
\begin{equation}
	q_i^{n+1}=\frac{1}{\Delta x_i}\int_{x_{i-1/2}}^{x_{i+1/2}}q^{appr}(t_{n+1}-0,x)\,dx.
	\label{eq:updateriemann}
\end{equation}
Notice that this is only in this averaging procedure that the choice of the pseudo-conservative variable
$q$ is involved. We can follow the computations of Section 2.3 in \cite{bouchut-2004}, the only difference being
that here the system is nonconservative. We deduce that
\begin{equation}
	q_i^{n+1}=q_i^n-\frac{\Delta t}{\Delta x_i}\left(\F_l(q_i^n,q_{i+1}^n)-\F_r(q_{i-1}^n,q_i^n)\right),
	\label{eq:updatediscr}
\end{equation}
where
\begin{equation}\begin{array}{l}
	\dsp \F_l(q_l,q_r)=F(q_l)-\int_{-\infty}^0\Bigl(R(\xi,q_l,q_r)-q_l\Bigr)d\xi,\\
	\dsp \F_r(q_l,q_r)=F(q_r)+\int_0^\infty\Bigl(R(\xi,q_l,q_r)-q_r\Bigr)d\xi,
	\label{eq:FlFr}
	\end{array}
\end{equation}
and the pseudo-conservative flux is
\begin{equation}
	F(q)=(hu_x^0,h(u_x^0)^2+P,h\sigma_{xx}u_x^0,h\sigma_{zz}u_x^0).
	\label{eq:pseudoconsflux}
\end{equation}
In \eqref{eq:pseudoconsflux}, the two last components are chosen arbitrarily, since
anyway the contributions of $F$ in \eqref{eq:updatediscr} cancel out.

Since the two first components of the system \eqref{eq:relax} are conservative,
the classical computations in this context give that for these two components,
the left and right numerical fluxes of \eqref{eq:FlFr} are equal and indeed take the value
of the flux of \eqref{eq:relax}, i.e. $hu_x^0$ and $h(u_x^0)^2+\pi$, at $x/t=0$.

We can notice that while solving the relaxation system \eqref{eq:relax}, the variables
$h$, $s_{xx}$ and $s_{zz}$ remain positive if they are initially (indeed this is subordinate to
the existence of a solution with positive $h$, which is seen below via explicit formulas
and under suitable choice for $c_l$, $c_r$). By the relation \eqref{eq:bshom}
this is also the case for $\sigma_{xx}$ and $\sigma_{zz}$.
Therefore, the invariant domain $\mathcal{U}$ in \eqref{eq:invariantdomain} is preserved by the numerical
scheme \eqref{eq:updatediscr}, this follows from the average formula \eqref{eq:updateriemann}
and the fact that $\mathcal{U}$ is convex (in the variable $q$).

\begin{remark} The above scheme satisfies the maximum principle on the variable $s_{xx}$,
and the minimum principle on the variable $s_{zz}$. This means that if initially one has
$s_{xx}\leq k$ for some constant $k>0$ (respectively $s_{zz}\geq k$), then
it remains true for all times.

This can be seen by observing that the set where $s_{xx}\leq k$ (respectively $s_{zz}\geq k$)
is convex in the variable $q$, because according to \eqref{eq:bshom}, \eqref{eq:variables},
it can be written as $q_1q_3\geq k^{-2}$ (respectively $k^2q_1^3-q_4\leq 0$).
Then, $\bs$ is just transported during the resolution of \eqref{eq:relax}, while the averaging
procedure \eqref{eq:updateriemann} preserves the convex sets.
Another proof is to write a discrete entropy inequality for an entropy $h\phi(s_{xx})$, which
is convex if $0\le\phi'\le s_{xx}\phi''$, take for example
$\phi(s_{xx})=\max(0,s_{xx}-k)^2/2$ (respectively for an entropy $h\phi(s_{zz})$, which
is convex if $0\le-\phi'\le 3s_{zz}\phi''$, take for example
$\phi(s_{zz})=k^{-1/3}s_{zz}-\frac{3}{2}s_{zz}^{2/3}+\frac{1}{2}k^{2/3}$
for $s_{zz}\leq k$, $\phi(s_{zz})=0$ for $s_{zz}\geq k$). We shall not write down the details
of this alternative proof.
\label{Remark maxprinc}
\end{remark}

\subsection{Energy inequality}
\label{sec:energyineq}
We define in a similar way the left and right numerical energy fluxes
\begin{equation}\begin{array}{l}
	\dsp \G_l(q_l,q_r)=G(q_l)-\int_{-\infty}^0\Bigl(E\bigl(R(\xi,q_l,q_r)\bigr)-E\bigl(q_l\bigr)\Bigr)d\xi,\\
	\dsp \G_r(q_l,q_r)=G(q_r)+\int_0^\infty\Bigl(E\bigl(R(\xi,q_l,q_r)\bigr)-E\bigl(q_r\bigr)\Bigr)d\xi,
	\label{eq:GlGr}
	\end{array}
\end{equation}
where $E$ is the energy of \eqref{eq:energy} without the topographic term $gbh$,
and
\begin{equation}
	G=(E+P)u_x^0
	\label{eq:G}
\end{equation}
is the energy flux. We have from \cite{bouchut-2004} that
a sufficient condition for the scheme to be energy satisfying is that
\begin{equation}
	\G_r(q_l,q_r)-\G_l(q_l,q_r)\leq 0.
	\label{eq:condentrdiscr}
\end{equation}
When this is satisfied, because of the convexity of $E$ with respect to $q$
one has the discrete energy inequality
\begin{equation}
	E(q_i^{n+1})-E(q_i^n)+\frac{\Delta t}{\Delta x_i}\Bigl(\G(q_i^n,q_{i+1}^n)-\G(q_{i-1}^n,q_i^n)\Bigr)\leq 0,
	\label{eq:discenineq}
\end{equation}
where the numerical energy flux $\G(q_l,q_r)$ is any function satisfying $\G_r(q_l,q_r)\leq \G(q_l,q_r)\leq \G_l(q_l,q_r)$.

In order to analyze the condition \eqref{eq:condentrdiscr}, let us introduce the internal energy $e(q)\geq 0$ by
\begin{equation}
	e=g\frac{h}{2}
+ \frac{\eta_p}{4\lambda}\bigl(\sigma_{xx}+\sigma_{zz}-\ln(\sigma_{xx}\sigma_{zz})-2\bigr),
	\label{eq:defe}
\end{equation}
so that
\begin{equation}
	E=h(u_x^0)^2/2+he,
	\label{eq:Ee}
\end{equation}
and $(\pds{h}e)_{|\bs}=P/h^2$. Then, while solving the relaxation system \eqref{eq:relax},
we solve simultaneously the equation for a new variable $\widehat e$,
\begin{equation}
	\pds{t}(\widehat e-\pi^2/2c^2) + u_x^0\pds{x}(\widehat e-\pi^2/2c^2) = 0,
	\label{eq:eqwidehate}
\end{equation}
where $\widehat e$ has left and right initial data $e(q_l)$ and $e(q_r)$. The reason
for writing \eqref{eq:eqwidehate} is that combining it with \eqref{eq:relax} yields
\begin{equation}
	\pds{t}\Bigl(h(u_x^0)^2/2+h\widehat e\Bigr)+\pds{x}\Bigl(\bigl(h(u_x^0)^2/2+h\widehat e+\pi\bigr)u_x^0\Bigr)=0.
	\label{eq:kineentropy}
\end{equation}
Define now
\begin{equation}
	\G(q_l,q_r)=\Bigl(\bigl(h(u_x^0)^2/2+h\widehat e+\pi\bigr)u_x^0\Bigr)_{x/t=0}.
	\label{eq:numentrfluxc}
\end{equation}
\begin{lemma} If for all values of $x/t$ the solution to \eqref{eq:relax}, \eqref{eq:eqwidehate} satisfies
\begin{equation}
	\widehat e\geq e(q),
	\label{eq:eehat}
\end{equation}
where here $q=R(x/t,q_l,q_r)$, then $\G_r(q_l,q_r)\leq \G(q_l,q_r)\leq \G_l(q_l,q_r)$ and the discrete energy inequality \eqref{eq:discenineq} holds.
\label{lemma entropy}
\end{lemma}
\begin{proof} Since \eqref{eq:kineentropy} is a conservative equation, one has
\begin{equation}\begin{array}{l}
	\dsp \G(q_l,q_r)=G(q_l)-\int_{-\infty}^0\Bigl(\left(h(u_x^0)^2/2+h\widehat e\right)(\xi)-E(q_l)\Bigr)d\xi\\
	\dsp\hphantom{\G(q_l,q_r)}
	=G(q_r)+\int_0^\infty\Bigl(\left(h(u_x^0)^2/2+h\widehat e\right)(\xi)-E(q_r)\Bigr)d\xi.
	\label{eq:idGc}
	\end{array}
\end{equation}
Therefore, comparing to \eqref{eq:GlGr}, we see that in order to get the result it is enough that
for all $\xi$
\begin{equation}
	E(R(\xi,q_l,q_r))\leq\left(h(u_x^0)^2/2+h\widehat e\right)(\xi),
	\label{eq:ineqproj}
\end{equation}
which is \eqref{eq:eehat}.
\end{proof}
In order to go further, we fix the following notation: in the solution to the Riemann
problem for \eqref{eq:relax}, there are three waves and two intermediate states,
denoted respectively by indices $l,*$ and $r,*$. Then we have the following
sufficient subcharacteristic condition (recall that $\pds{h}P|_{\bs}$ is given by \eqref{eq:dPdh}).
\begin{lemma}
If $c_l$, $c_r$ are chosen such that the heights $h_l^\star$, $h_r^\star$
are positive and satisfy
\begin{equation}
\begin{split}
\forall h\in[h_l,h_l^\star] \quad
h^2 \pds{h}P|_{\bs}(h,\bs_l) \le c_l^2,
\\
\forall h\in[h_r,h_r^\star] \quad
h^2 \pds{h}P|_{\bs}(h,\bs_r) \le c_r^2,
\end{split} 
\end{equation}
then \eqref{eq:eehat} holds and thus the discrete energy inequality \eqref{eq:discenineq} is valid.
\label{lemma:subchar}
\end{lemma}
\begin{proof}
The arguments of decomposition in elementary dissipation terms along the waves
used in Lemma 2.20 in~\cite{bouchut-2004} can be checked to apply
without modification.
\end{proof}
\begin{lemma} Denote
\begin{equation}
	P_l=P(h_l,\bs_l),\quad P_r=P(h_r,\bs_r),\qquad
	a_l=\sqrt{\pds{h}P|_{\bs}(h_l,\bs_l) }, \quad a_r=\sqrt{\pds{h}P|_{\bs}(h_r,\bs_r) },
	\label{eq:lrsoundspeeds}
\end{equation}
and define the relaxation speeds $c_l$, $c_r$ by
\begin{equation}
	\begin{array}{l}
	\dsp\frac{c_l}{h_l} =  a_l+ 2\left(\max\Bigl(0,u^0_{x,l} - u^0_{x,r}\Bigr) + \frac{\max\Bigl(0,P_r-P_l\Bigr)}{h_l a_l + h_r a_r}\right),\\
	\dsp \frac{c_r}{h_r}=a_r+2\left(\max\Bigl(0,u^0_{x,l} - u^0_{x,r}\Bigr) + \frac{\max\Bigl(0,P_l-P_r\Bigr)}{h_l a_l + h_r a_r}\right).
	\end{array}
	\label{eq:cformulae}
\end{equation}
Then the positivity and subcharacteristic conditions of Lemma \ref{lemma:subchar} are
satisfied, and the discrete energy inequality \eqref{eq:discenineq} holds.
\label{lemma:speeds}
\end{lemma}
\begin{proof} From \eqref{eq:dPdh} and \eqref{eq:bshom} we have
\begin{equation}
	\pds{h}P|_{\bs}=gh+\frac{\eta_p}{2\lambda}\Bigl(3(hs_{zz})^2+\frac{1}{(hs_{xx})^2}\Bigr).
	\label{eq:pPh}
\end{equation}
Denoting $\varphi(h,\bs)=h\sqrt{\pds{h}P|_{\bs}}$, we compute
\begin{equation}\begin{array}{l}
	\dsp \pds{h}\varphi|_{\bs}=\sqrt{\pds{h}P|_{\bs}}
	+\frac{h}{2\sqrt{\pds{h}P|_{\bs}}}\left(g+\frac{\eta_p}{2\lambda}\Bigl(
	6hs_{zz}^2-\frac{2}{h^3s_{xx}^2}\Bigr)\right)\\
	\dsp\hphantom{\pds{h}\varphi|_{\bs}}
	=\frac{1}{2\sqrt{\pds{h}P|_{\bs}}}\left(2gh+\frac{\eta_p}{\lambda}\Bigl(3(hs_{zz})^2+\frac{1}{(hs_{xx})^2}\Bigr)
	+gh+\frac{\eta_p}{2\lambda}\Bigl(
	6(hs_{zz})^2-\frac{2}{(hs_{xx})^2}\Bigr)\right)\\
	\dsp\hphantom{\pds{h}\varphi|_{\bs}}
	=\frac{1}{2\sqrt{\pds{h}P|_{\bs}}}\left(3gh+6\frac{\eta_p}{\lambda}(hs_{zz})^2\right).
	\label{eq:pdsphi}
	\end{array}
\end{equation}
Therefore, we deduce that $\varphi$ satisfies
\begin{equation}\begin{array}{c}
	\dsp \pds{h}\varphi|_{\bs}>0,\\
	\dsp\varphi(h,\bs)\rightarrow\infty\quad\mbox{as }h\rightarrow\infty,\\
	\dsp \pds{h}\varphi|_{\bs}\leq 2\sqrt{\pds{h}P|_{\bs}}.
	\label{eq:proppress}
	\end{array}
\end{equation}
Following [Proposition 3.2]~\cite{bouchut-klingenberg-waagan-2010}
with $\alpha=2$, we get the result.
\end{proof}
\begin{remark}[Bounds on the propagation speeds] Lemma \ref{lemma:speeds} is also
valid with the formulas of  [Proposition 2.18]~\cite{bouchut-2004} instead of \eqref{eq:cformulae}.
Here we prefer \eqref{eq:cformulae} because in the context of possibly negative pressure $P$
these formulas ensure the following estimate on the propagation speeds:
\begin{equation}
	\max\left(\frac{c_l}{h_l},\frac{c_r}{h_r}\right)
	\leq C\left(|u_{x,l}^0|+|u_{x,r}^0|+a_l+a_r\right),
	\label{eq:estspeeds}
\end{equation}
with $C$ an absolute constant. This follows from the property that $|P|\leq  h\pds{h}P|_{\bs}$,
which is seen on \eqref{eq:P}-\eqref{eq:dPdh}.
\label{Remark speeds}
\end{remark}

\subsection{Numerical fluxes and CFL condition}
\label{sec:numflux}
The Riemann problem for the relaxation system \eqref{eq:relax}, \eqref{eq:eqwidehate}
has to be solved with initial data $q_l$, $q_r$ completed with \eqref{eq:pilpir},
the relation \eqref{eq:bshom}, $\widehat e_l=e(q_l)\equiv e_l$,   $\widehat e_r=e(q_r)\equiv e_r$,
and \eqref{eq:lrsoundspeeds}, \eqref{eq:cformulae}.
The explicit solution is given, according to \cite{bouchut-2004}, by the following formulae.
It has three waves speeds $\Sigma_1<\Sigma_2<\Sigma_3$,
\begin{equation}
	\Sigma_1=u_{x,l}^0-c_l/h_l,\qquad\Sigma_2=u_{x,*}^0,\qquad\Sigma_3=u_{x,r}^0+c_r/h_r,
	\label{eq:Sigmas}
\end{equation}
and the variables take the value "l" for $x/t<\Sigma_1$, 
"l*" for $\Sigma_1<x/t<\Sigma_2$, "r*" for $\Sigma_2<x/t<\Sigma_3$, "r" for $\Sigma_3<x/t$.
The "l*" and "r*" values are given by
\begin{equation}\begin{array}{c}
	\dsp (u_x^0)_l^*=(u_x^0)_r^*=u_{x,*}^0=\frac{c_lu_{x,l}^0+c_ru_{x,r}^0+\pi_l-\pi_r}{c_l+c_r},\qquad
	\pi_l^*=\pi_r^*=\frac{c_r\pi_l+c_l\pi_r-c_lc_r(u_{x,r}^0-u_{x,l}^0)}{c_l+c_r},\vphantom{\Biggl|}\\
	\dsp \frac{1}{h_l^*}=\frac{1}{h_l}+\frac{c_r(u_{x,r}^0-u_{x,l}^0)+\pi_l-\pi_r}{c_l(c_l+c_r)},
	\qquad\frac{1}{h_r^*}=\frac{1}{h_r}+\frac{c_l(u_{x,r}^0-u_{x,l}^0)+\pi_r-\pi_l}{c_r(c_l+c_r)},\vphantom{\Biggl|}
	\label{eq:l*r*}
	\end{array}
\end{equation}
\begin{equation}
	c_l^*=c_l,\quad c_r^*=c_r,\quad \bs_l^*=\bs_l,\quad \bs_r^*=\bs_r,
	\label{eq:sl*sr*}
\end{equation}
\begin{equation}
	\sigma_{xx,l}^*=\sigma_{xx,l}\left(\frac{h_l}{h_l^*}\right)^2,\ 
	\sigma_{xx,r}^*=\sigma_{xx,r}\left(\frac{h_r}{h_r^*}\right)^2,\ 
	\sigma_{zz,l}^*=\sigma_{zz,l}\left(\frac{h_l^*}{h_l}\right)^2,\ 
	\sigma_{zz,r}^*=\sigma_{zz,r}\left(\frac{h_r^*}{h_r}\right)^2,
	\label{eq:sigm*}
\end{equation}
\begin{equation}
	\widehat e_l^*=e_l-\frac{(\pi_l)^2}{2c_l^2}+\frac{(\pi_l^*)^2}{2c_l^2},\qquad
	\widehat e_r^*=e_r-\frac{(\pi_r)^2}{2c_r^2}+\frac{(\pi_r^*)^2}{2c_r^2}.
	\label{eq:el*er*}
\end{equation}
Then we need to compute the left/right numerical fluxes \eqref{eq:FlFr} that are involved
in the update formula \eqref{eq:updatediscr}. Since the $h$ and $hu_x^0$ components
in \eqref{eq:relax} are conservative, classical computations give the associated numerical fluxes,
and we have
\begin{equation}
	\F_l=\Bigl(\F^h,\F^{hu_x^0},\F_l^{h\sigma_{xx}},\F_l^{h\sigma_{zz}}\Bigr),\qquad
	\F_r=\Bigl(\F^h,\F^{hu_x^0},\F_r^{h\sigma_{xx}},\F_r^{h\sigma_{zz}}\Bigr),
	\label{eq:lrfluxes}
\end{equation}
where the conservative part involves the Riemann solution evaluated at $x/t=0$,
\begin{equation}
	\F^h=(hu_x^0)_{x/t=0},\qquad \F^{hu_x^0}=(h(u_x^0)^2+\pi)_{x/t=0}.
	\label{eq:h,hunumfluxes}
\end{equation}
More explicitly, \eqref{eq:h,hunumfluxes} yields that the quantities between parenthese
are evaluated at "l" if $\Sigma_1\geq 0$, at "l*" if $\Sigma_1\leq 0\leq\Sigma_2$,
at "r*" if $\Sigma_2\leq 0\leq \Sigma_3$, and at "r" if $\Sigma_3\leq 0$. As usual there is no
ambiguity in the resulting value when equality occurs in these conditions.
The numerical energy flux \eqref{eq:numentrfluxc} involved in \eqref{eq:discenineq}
can be computed in the same way.

We complete these formulas by computing the left/right numerical fluxes for the variables
$h\sigma_{xx}$, $h\sigma_{zz}$ from \eqref{eq:FlFr},
\begin{equation}\begin{array}{l}
	\dsp \F_l^{h\sigma_{xx}}
	=(h\sigma_{xx}u_x^0)_l+\min(0,\Sigma_1)\Bigl((h\sigma_{xx})_l^*-(h\sigma_{xx})_l\Bigr)\\
	\dsp\mkern 150mu
	+\min(0,\Sigma_2)\Bigl((h\sigma_{xx})_r^*-(h\sigma_{xx})_l^*\Bigr)
	+\min(0,\Sigma_3)\Bigl((h\sigma_{xx})_r-(h\sigma_{xx})_r^*\Bigr),
	\label{eq:Fnonconsl}
	\end{array}
\end{equation}
\begin{equation}\begin{array}{l}
	\dsp \F_r^{h\sigma_{xx}}
	=(h\sigma_{xx}u_x^0)_r-\max(0,\Sigma_1)\Bigl((h\sigma_{xx})_l^*-(h\sigma_{xx})_l\Bigr)\\
	\dsp\mkern 150mu
	-\max(0,\Sigma_2)\Bigl((h\sigma_{xx})_r^*-(h\sigma_{xx})_l^*\Bigr)
	-\max(0,\Sigma_3)\Bigl((h\sigma_{xx})_r-(h\sigma_{xx})_r^*\Bigr),
	\label{eq:Fnonconsr}
	\end{array}
\end{equation}
the $h\sigma_{zz}$ fluxes being computed with the same formulas, replacing
"xx" by "zz".

The maximal propagation speed is then
\begin{equation}
	A(q_l,q_r)=\max(|\Sigma_1|,|\Sigma_2|,|\Sigma_3|),
	\label{eq:maxspeed}
\end{equation}
and the CFL condition \eqref{eq:CFLhalfsolver} becomes
\begin{equation}
	\Delta t A(q_i,q_{i+1})\leq\frac{1}{2}\min(\Delta x_i,\Delta x_{i+1}).
	\label{eq:CFLhalf}
\end{equation}
Not that with \eqref{eq:estspeeds} and \eqref{eq:Sigmas} we get
\begin{equation}
	A(q_l,q_r)\leq C\left(|u_{x,l}^0|+|u_{x,r}^0|+a_l+a_r\right)
	\label{eq:boundspeed}
\end{equation}
with $C$ an absolute constant, bounding the propagation speeds of the approximate Riemann solver
whenever the left and right true speeds remain bounded.
This property is more general than the possibility of treating data with vacuum
considered in \cite{bouchut-2004}.

We have obtained finally the following theorem.
\begin{theorem} Consider the system \eqref{eq:pseudocons} with the pressure law \eqref{eq:Ppseudo},
and denote the pseudo-conservative variable by $q=(h,hu^0_x,h\sigma_{xx},h\sigma_{zz})$.
Under the CFL condition \eqref{eq:CFLhalf},
the scheme \eqref{eq:updatediscr} with the numerical fluxes $\F_l(q_l,q_r)$, $\F_r(q_l,q_r)$
defined above via \eqref{eq:lrfluxes}, and with the choice of the speeds \eqref{eq:lrsoundspeeds}, \eqref{eq:cformulae},
satisfies the following properties.

\noindent (i) It is consistent with \eqref{eq:pseudocons}-\eqref{eq:Ppseudo} for smooth solutions,

\noindent (ii) It keeps the positivity of $h$, $\sigma_{xx}$, $\sigma_{zz}$,

\noindent (iii) It is conservative in the variables $h$ and $hu_x^0$,

\noindent (iv) It satisfies the discrete energy inequality \eqref{eq:discenineq},

\noindent (v) It satisfies the maximum principle on the variable $s_{xx}$,
and the minimum principle on the variable $s_{zz}$,

\noindent (vi) Steady contact discontinuities where $u_x^0=0$, $P=cst$ are exactly resolved,

\noindent (vii) Data with bounded propagation speeds give finite numerical propagation speed.

\noindent (viii) The numerical viscosity is sharp, in the sense that the propagation speeds
$\Sigma_i$ of the approximate Riemann solver tend to the exact propagation speeds when
the left and right states $q_l$, $q_r$ tend to a common value.
\label{theorem:scheme}
\end{theorem}

\subsection{Topography treatment}
\label{sec:topo}
Consider now our system \eqref{eq:reducedmodelpseudocons} with topography,
but without  the relaxation source terms, i.e.
\begin{equation}
\label{eq:systemtopo}
\left\lbrace
\begin{aligned}
& \pds{t}h + \pds{x}(h u_x^0) = 0,
\\
& \pds{t}(h u_x^0) + \pds{x}\left( h(u_x^0)^{2} + g\frac{h^2}2 + \frac{\eta_p}{2\lambda}h(\sigma_{zz}-\sigma_{xx}) \right) = 
-gh\pds{x}b,
\\
& \pds{t}(h\sigma_{xx}) + \pds{x}(h\sigma_{xx}u_x^0)-2h\sigma_{xx}\pds{x}u_x^0
=  0,
\\
& \pds{t}(h\sigma_{zz}) + \pds{x}(h\sigma_{zz}u_x^0)+2h\sigma_{zz}\pds{x}u_x^0
=  0.
\end{aligned}
\right.
\end{equation}
With respect to the previous sections, the term $-gh\pds{x}b$ has been put back,
where the topography is a given function $b(x)$. For \eqref{eq:systemtopo}, the energy inequality
\eqref{eq:estimate2} is modified only by the fact that there is no right-hand side.
Thus it can be written
\begin{equation}
	\pds{t}\widetilde E+\pds{x}\widetilde G\leq 0,
	\label{eq:enerb}
\end{equation}
with
\begin{equation}
	\widetilde E(q,b)=E(q)+ghb,\qquad \widetilde G(q,b)=G(q)+ghbu_x^0,
	\label{eq:tildeEG}
\end{equation}
where $E$ and $G$ are given by \eqref{eq:Ee}, \eqref{eq:defe}, \eqref{eq:G}.
Recall that the steady states at rest of Remark \ref{rem:steady} are defined by
\begin{equation}
	u_x^0=0,\quad h+b=cst,\quad\sigma_{xx}=\sigma_{zz}=1.
	\label{eq:rest}
\end{equation}
Our scheme for \eqref{eq:systemtopo} is written as
\begin{equation}
	q_i^{n+1}=q_i^n-\frac{\Delta t}{\Delta x_i}\left(F_l(q_i^n,q_{i+1}^n,\Delta b_{i+1/2})-F_r(q_{i-1}^n,q_i^n,\Delta b_{i-1/2})\right),
	\label{eq:updatediscrhr}
\end{equation}
where as before $q=(h,hu_x^0,h\sigma_{xx},h\sigma_{zz})$, and $\Delta b_{i+1/2}=b_{i+1}-b_i$.
Thus we need to define the left and right numerical fluxes
$F_l(q_l,q_r,\Delta b)$, $F_r(q_l,q_r,\Delta b)$, for all left and right values
$q_l$, $q_r$, $b_l$, $b_r$, with $\Delta b=b_r-b_l$.
We use the hydrostatic reconstruction method of \cite{audusse-bouchut-bristeau-klein-perthame-2004} (see also \cite{bouchut-morales-2010}),
and define
\begin{equation}
	h_l^\sharp=\bigl(h_l-(\Delta b)_+\bigr)_+,\qquad h_r^\sharp=\bigl(h_r-(-\Delta b)_+\bigr)_+,
	\label{eq:rech}
\end{equation}
\begin{equation}
	q_l^\sharp=\Bigl(h_l^\sharp,h_l^\sharp u_{x,l}^0,h_l^\sharp\sigma_{xx,l},h_l^\sharp\sigma_{zz,l}\Bigr),\qquad
	q_r^\sharp=\Bigl(h_r^\sharp,h_r^\sharp u_{x,r}^0,h_r^\sharp\sigma_{xx,r},h_r^\sharp\sigma_{zz,r}\Bigr),
	\label{eq:qlrsharp}
\end{equation}
with the notation $x_+\equiv\max(0,x)$. Note that we use the notation $\sharp$ instead
of $*$ in order to avoid confusions with the intermediate states of the Riemann solver
of the previous sections. Then the numerical fuxes are defined by
\begin{equation}\begin{array}{c}
	\dsp F_l(q_l,q_r,\Delta b)=\F_l(q_l^\sharp,q_r^\sharp)+\biggl(0,g\frac{h_l^2}{2}-g\frac{h_l^{\sharp 2}}{2},0,0\biggr),\\
	\dsp F_r(q_l,q_r,\Delta b)=\F_r(q_l^\sharp,q_r^\sharp)+\biggl(0,g\frac{h_r^2}{2}-g\frac{h_r^{\sharp 2}}{2},0,0\biggr),
	\label{eq:Flr}
	\end{array}
\end{equation}
where $\F_l$ and $\F_r$ are the numerical fluxes \eqref{eq:lrfluxes} of the problem without topography.
\begin{theorem} The scheme \eqref{eq:updatediscrhr}
with the numerical fluxes $F_l$, $F_r$ defined by \eqref{eq:Flr}, \eqref{eq:rech}, \eqref{eq:qlrsharp}
satisfies the following properties.

\noindent (i) It is consistent with \eqref{eq:systemtopo} for smooth solutions,

\noindent (ii) It keeps the positivity of $h$, $\sigma_{xx}$, $\sigma_{zz}$ under the CFL condition
$\Delta t A(q_l^\sharp,q_r^\sharp)\leq\frac{1}{2}\min(\Delta x_l,\Delta x_r)$ with
$A$ defined by \eqref{eq:maxspeed},

\noindent (iii) It is conservative in the variable $h$,

\noindent (iv) It satisfies a semi-discrete energy inequality associated to \eqref{eq:enerb},

\noindent (v) It is well-balanced, i.e. preserves the steady states at rest \eqref{eq:rest}.
\label{theorem:wbscheme}
\end{theorem}
\begin{proof} We ommit the proof of the points (i) to (iii), which follow the proof of Proposition 4.14 in \cite{bouchut-2004}.

\noindent For the proof of (v), consider data $q_l$, $q_r$, $b_l$, $b_r$ at rest, i.e.
satisfying $u_{x,l}^0=u_{x,r}^0=0$, $h_l+b_l=h_r+b_r$, $\sigma_{xx,l}=\sigma_{xx,r}=\sigma_{zz,l}=\sigma_{zz,r}=1$.
Then from \eqref{eq:rech}, \eqref{eq:qlrsharp} we get $q_l^\sharp=q_r^\sharp$, the common value $q^\sharp$ being $q_r$ if $\Delta b\geq 0$,
or $q_l$ if $\Delta b\leq 0$.  We observe that then
$\F_l(q_l^\sharp,q_r^\sharp)=\F_r(q_l^\sharp,q_r^\sharp)=F(q^\sharp)$ with $F$ given in \eqref{eq:pseudoconsflux},
and that indeed $F(q^\sharp)=(0,gh^{\sharp 2}/2,0,0)$.
The formulas \eqref{eq:Flr} yield $F_l=(0,gh_l^2/2,0,0)=F(q_l)$, $F_r=(0,gh_r^2/2,0,0)=F(q_r)$.
If this is true at all interfaces, \eqref{eq:updatediscrhr} gives $q_i^{n+1}=q_i^n$, which proves the claim.

Let us finally prove (iv). First, the scheme without topography satisfies the discrete energy
inequality \eqref{eq:discenineq}. According to \cite{bouchut-2004} section 2.2.2, it implies
the semi-discrete energy inequality, characterized by
\begin{equation}\begin{array}{c}
	\dsp G(q_r)+E'(q_r)(\F_r(q_l,q_r)-F(q_r))\leq\G(q_l,q_r),\\
	\dsp \G(q_l,q_r)\leq G(q_l)+E'(q_l)(\F_l(q_l,q_r)-F(q_l)),
	\label{eq:semihom}
	\end{array}
\end{equation}
for all values of $q_l$, $q_r$, and where $E'$ is the derivative of $E$ with respect to $q$.
Then, for the scheme with topography, the characterization of the semi-discrete
energy inequality writes
\begin{equation}\begin{array}{c}
	\dsp \widetilde G(q_r,b_r)+\widetilde E'(q_r,b_r)(F_r-F(q_r))\leq \widetilde \G(q_l,q_r,b_l,b_r),\\
	\dsp \widetilde \G(q_l,q_r,b_l,b_r)\leq\widetilde G(q_l,b_l)+\widetilde E'(q_l,b_l)(F_l-F(q_l)),
	\label{eq:semiwb}
	\end{array}
\end{equation}
where $\widetilde E$ and $\widetilde G$ are defined by \eqref{eq:tildeEG},
$\widetilde E'$ denotes the derivative of $\widetilde E$ with respect to $q$,
and $\widetilde \G$ is an unknown consistent numerical entropy flux.
Let us choose
\begin{equation}
	\widetilde \G(q_l,q_r,b_l,b_r)=\G(q_l^\sharp,q_r^\sharp)+\F^h(q_l^\sharp,q_r^\sharp)gb^\sharp,
	\label{eq:numentrfluxwb}
\end{equation}
where $\F^h$ is the common $h-$component of $\F_l$ and $\F_r$, and for some $b^\sharp$ that is defined below.
Then, noticing that $\widetilde E'(q,b)=E'(q)+gb(1,0,0,0)$, we can write
the desired inequalities \eqref{eq:semiwb} as
\begin{equation}\begin{array}{c}
	\dsp G(q_r)+E'(q_r)(F_r-F(q_r))+\F^h(q_l^\sharp,q_r^\sharp)gb_r
	\leq\G(q_l^\sharp,q_r^\sharp)+\F^h(q_l^\sharp,q_r^\sharp)gb^\sharp,\\
	\dsp \G(q_l^\sharp,q_r^\sharp)+\F^h(q_l^\sharp,q_r^\sharp)gb^\sharp
	\leq G(q_l)+E'(q_l)(F_l-F(q_l))+\F^h(q_l^\sharp,q_r^\sharp)gb_l.
	\label{eq:condsentr1}
	\end{array}
\end{equation}
But using \eqref{eq:semihom} evaluated at $q_l^\sharp$, $q_r^\sharp$ and comparing the result
with \eqref{eq:condsentr1}, we get the sufficient conditions
\begin{equation}\begin{array}{c}
	\dsp G(q_r)+E'(q_r)(F_r-F(q_r))+\F^h(q_l^\sharp,q_r^\sharp)gb_r
	\leq G(q_r^\sharp)+E'(q_r^\sharp)(\F_r(q_l^\sharp,q_r^\sharp)-F(q_r^\sharp))+\F^h(q_l^\sharp,q_r^\sharp)gb^\sharp,\\
	\dsp G(q_l^\sharp)+E'(q_l^\sharp)(\F_l(q_l^\sharp,q_r^\sharp)-F(q_l^\sharp))+\F^h(q_l^\sharp,q_r^\sharp)gb^\sharp
	\leq G(q_l)+E'(q_l)(F_l-F(q_l))+\F^h(q_l^\sharp,q_r^\sharp)gb_l.
	\label{eq:condsentr2}
	\end{array}
\end{equation}
We compute now
\begin{equation}
	E'(q)=\biggl(-\frac{(u_x^0)^2}{2}+gh-\frac{\eta_p}{4\lambda}\ln(\sigma_{xx}\sigma_{zz}),u_x^0,
	\frac{\eta_p}{4\lambda}(1-1/\sigma_{xx}),\frac{\eta_p}{4\lambda}(1-1/\sigma_{zz})\biggr),
	\label{eq:E'}
\end{equation}
and writing
\begin{equation}\begin{array}{c}
	\dsp F(q)=\biggl(hu_x^0,h(u_x^0)^2+g\frac{h^2}{2}+\frac{\eta_p}{2\lambda}h(\sigma_{zz}-\sigma_{xx}),h\sigma_{xx}u_x^0,h\sigma_{zz}u_x^0\biggr),\\
	\dsp G(q)=\biggl(h\frac{(u_x^0)^2}{2}+gh^2+\frac{\eta_p}{4\lambda}h(\sigma_{xx}+\sigma_{zz}-\ln(\sigma_{xx}\sigma_{zz})-2)
	+\frac{\eta_p}{2\lambda}h(\sigma_{zz}-\sigma_{xx})\biggr)u_x^0,
	\label{eq:FG}
	\end{array}
\end{equation}
we deduce the identity
\begin{equation}
	G(q)-E'(q)F(q)=-g\frac{h^2}{2}u_x^0.
	\label{eq:idGEF}
\end{equation}
Thus the inequality \eqref{eq:condsentr2} simplifies to
\begin{equation}\begin{array}{c}
	\dsp -g\frac{h_r^2}{2}u_{x,r}^0+E'(q_r)F_r+\F^h(q_l^\sharp,q_r^\sharp)gb_r
	\leq -g\frac{h_r^{\sharp 2}}{2}u_{x,r}^0+E'(q_r^\sharp)\F_r(q_l^\sharp,q_r^\sharp)+\F^h(q_l^\sharp,q_r^\sharp)gb^\sharp,\\
	\dsp -g\frac{h_l^{\sharp 2}}{2}u_{x,l}^0+E'(q_l^\sharp)\F_l(q_l^\sharp,q_r^\sharp)+\F^h(q_l^\sharp,q_r^\sharp)gb^\sharp
	\leq -g\frac{h_l^2}{2}u_{x,l}^0+E'(q_l)F_l+\F^h(q_l^\sharp,q_r^\sharp)gb_l.
	\label{eq:condsentr3}
	\end{array}
\end{equation}
Now, using \eqref{eq:Flr} and the fact that $E'(q_r)-E'(q_r^\sharp)=\bigl(g(h_r-h_r^\sharp),0,0,0\bigr)$,
$E'(q_l)-E'(q_l^\sharp)=\bigl(g(h_l-h_l^\sharp),0,0,0\bigr)$, the desired inequalities \eqref{eq:condsentr3}
rewrite
\begin{equation}\begin{array}{c}
	\dsp g\bigl(h_r-h_r^\sharp+b_r-b^\sharp\bigr)\F^h(q_l^\sharp,q_r^\sharp)\leq 0,\\
	\dsp g\bigl(h_l-h_l^\sharp-b^\sharp+b_l\bigr)\F^h(q_l^\sharp,q_r^\sharp)\geq 0.
	\label{eq:condsentr4}
	\end{array}
\end{equation}
We choose now $b^\sharp=\max(b_l,b_r)$, so that \eqref{eq:condsentr4} can be put in the form
\begin{equation}\begin{array}{c}
	\dsp \bigl(h_r-h_r^\sharp-(-\Delta b)_+\bigr)\F^h(q_l^\sharp,q_r^\sharp)\leq 0,\\
	\dsp \bigl(h_l-h_l^\sharp-(\Delta b)_+\bigr)\F^h(q_l^\sharp,q_r^\sharp)\geq 0.
	\label{eq:condsentr5}
	\end{array}
\end{equation}
Finally, taking into account \eqref{eq:rech}, we observe that if $h_l-(\Delta b)_+\geq 0$
then the second line of \eqref{eq:condsentr5} is an identity, otherwise $h_l^\sharp=0$ and
the the second inequality of \eqref{eq:condsentr5} holds because $\F^h(0,q_r^\sharp)\leq 0$
by the $h-$nonnegativity of the numerical flux. The same argument is valid for the first inequality
of \eqref{eq:condsentr5}, which concludes the proof.
\end{proof}
\begin{remark} The maximum principle property on $s_{xx}$ and minimum principle property
on $s_{zz}$, that hold for the solver without topography, are not valid for the above solver
with topography, even if it should hold at the continuous level.
\label{Remark maxprinchc}
\end{remark}

\subsection{Numerical results}

We now illustrate our model by numerical simulations performed with the scheme described above. 
Note that the model can be considered independently of its derivation and we explore numerical values beyond the physical regime of the Section~\ref{sec:derivation}.
We denote by $H(x)$ the Heaviside function with jump $+1$ at $x=0$.
For all numerical simulations, we chose Neumann conditions at boundary interfaces.

{\bf Test case 1.} It is a Riemann problem with initial condition
$(h,hu^0_x,h\sigma_{xx},h\sigma_{zz})(t=0)=(3-2H(x))(1,0,1,1)$,
without source term ($b\equiv0$), that can be interpreted as a ``dam'' break on a wet floor, 
with polymeric fluid initially at rest everywhere.
We first fix $\eta_p=\lambda=1$ and study the convergence of our scheme with respect to the spatial
discretization parameter for $50,100,200$ and $400$ points and a constant $CFL=1/2$.
The results at final time $T=.2$ are shown in Fig.~\ref{fig:test1convergence}.

\begin{figure}
 \includegraphics[scale=.7]{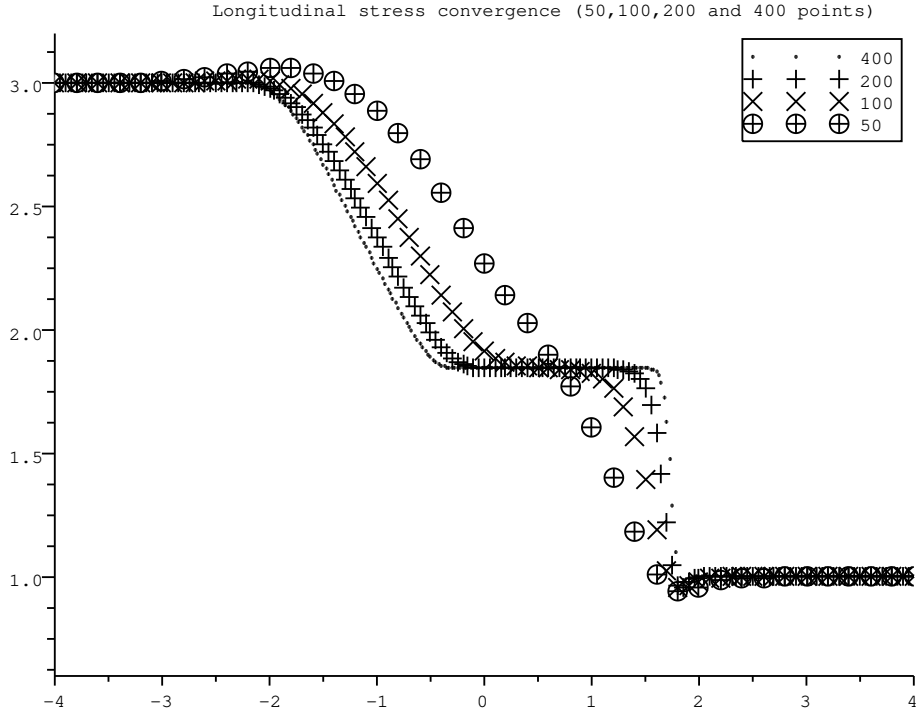}
\caption{\label{fig:test1convergence} Convergence of the discretized variables $h\sigma_{xx}$ in Test case 1}
\end{figure}

Then, using $400$ points and a constant $CFL=1/2$, we let $\eta_p$ vary as $\lambda=1$ is fixed.
The results are shown in Fig.~\ref{fig:test1variationseta1} and~\ref{fig:test1variationseta2}.
On the one hand, the limit case $\eta_p\to0$ coincides with the usual shallow-water model,
since then the pressure assumes the same values as in a relaxation scheme for the Saint-Venant equations
(independent of $h\sigma_{xx}$ and $h\sigma_{zz}$) while
$s_{xx},s_{zz}$ become passive tracers and their evolution is only one-way coupled
-- in fact enslaved -- to the autonomous dynamics of the Saint-Venant system of equations.
On the other hand, the case $\eta_p\to\infty$ is some kind of ``rigid limit''.
As expected from the formulae~\eqref{eq:eigenvalues} for the eigenvalues of the Jacobian matrix,
the left-going rarefaction wave and the right-going shock wave are all the faster as the viscosity $\eta_p$ increases
(we hardly see them at $T=.2$ for $\eta_p=10^{+3}$ in Fig.~\ref{fig:test1variationseta1} and~\ref{fig:test1variationseta2}).
This is consistent with the physical notion of rigid limit (sound waves are faster in solids than in liquids).
On the contrary, the intermediate wave (a right-going contact discontinuity) is all the {\it slower} as $\eta_p$ increases,
and the jump of $h$ across it is all the larger. This was not obvious to us at first.
It could be explained for instance by the fact that $c_l+c_r$ 
(at the denominator of the formulae for $u^0_{x,\star}$ and the two intermediate states $1/h^\star_l,1/h^\star_r$ in~\eqref{eq:l*r*}) 
becomes very large when $\eta_p$ increases.
Notice that the jumps for $\sigma_{xx}$ and $\sigma_{zz}$ are directly related to that for $h$ through~\eqref{eq:sigm*} 
and are consistently small when that for $h$ is large.
In any case, the materials on the left of the intermediate wave (where $h$ is higher) is stretched in direction $\be_x$ and compressed in direction $\be_z$,
while it is stretched in direction $\be_z$ and compressed in direction $\be_x$ on the right,
which we can interpret as the manifestation of a close-to-equilibrium stability property (see Section~\ref{sec:conclusion}).

\begin{figure}
 \include{ucm0159IHeta1ep0} 
 \include{ucm0159IUeta1ep0} 
\caption{\label{fig:test1variationseta1} Variations of the variables $h,u^0_{x}$ with $\eta_p$ in Test case 1}
\end{figure}

\begin{figure}
 \include{ucm0159IXeta1ep0} 
 \include{ucm0159IZeta1ep0}  
\caption{\label{fig:test1variationseta2} Variations of the variables $\sigma_{xx},\sigma_{zz}$ with $\eta_p$ in Test case 1}
\end{figure}

We also let $\lambda$ vary as $\eta=1$ is fixed. 
The results (still at same given time $T=.2$) are shown in Fig.~\ref{fig:test1variationslambda1} and~\ref{fig:test1variationslambda2}.
One clearly sees here the competition between transport and diffusion of viscoelastic effects (the viscoelastic energy is stored in the new variables, that are transported but also diffused).
Indeed, the viscoelastic energy is dissipated all the more rapidly as the relaxation time $\lambda$ is small (the viscous, ``Low-Weissenberg'' limit).
And the waves are all the more smoothed as $\lambda$ is small (the source terms, which act as diffusive terms, are all the more important).
On the other hand, the jump across the contact discontinuity is all the {\sl smaller} for $\sigma_{xx}$ and $\sigma_{zz}$ as $\lambda$ is small,
but all the {\sl larger} for $h$. This is coherent with a reasoning similar to the one above when $\eta_p$ only was varied:
a smaller relaxation time $\lambda$ implies slightly faster rarefaction and shock waves because of~\eqref{eq:eigenvalues}, 
and a larger coefficient $c_l+c_r$ in the formulae~\eqref{eq:l*r*}. 

\begin{figure}
 \include{ucm0159IH1ep0lambda} 
 \include{ucm0159IU1ep0lambda} 
\caption{\label{fig:test1variationslambda1} Variations of the variables $h,u^0_{x}$ with $\lambda$ in Test case 1}
\end{figure}

\begin{figure}
 \include{ucm0159IX1ep0lambda} 
 \include{ucm0159IZ1ep0lambda}  
\caption{\label{fig:test1variationslambda2} Variations of the variables $\sigma_{xx},\sigma_{zz}$ with $\lambda$ in Test case 1}
\end{figure}

{\bf Test case 2.} It is a Riemann problem again without source term $b\equiv0$ but {\it with vacuum} in the initial condition $(h,hu^0_x,h\sigma_{xx},h\sigma_{zz})(t=0)=(3-3H(x))(1,0,1,1)$, 
which can be interpreted as a ``dam'' break on a dry floor.
The results in Fig.~\ref{fig:test2rest1} and Fig.~\ref{fig:test2rest2} at $T=.5$ show again that 
small $\lambda$ and large $\eta_p$ imply a fast right-going rarefaction wave and a slow contact discontinuity, 
with a large jump for $h$ and small jumps for $\sigma_{xx},\sigma_{zz}$ at contact discontinuity.
On the contrary, large $\lambda$ and small $\eta_p$ imply a slow right-going rarefaction wave and a fast contact discontinuity, with a small jump for $h$ and large jumps for $\sigma_{xx},\sigma_{zz}$ at contact discontinuity. 
We nevertheless note that the ``High-Weissenberg'' limit $\lambda\to+\infty$ (recall Remark~\ref{rem:limit}) 
is more difficult. 
In particular, although $\sigma_{zz}$ remains bounded, $\sigma_{xx}$ seems to become unbounded close to the wet/dry front.
Recalling our Remark~\ref{Remark maxprinc} (see also the comments below about the occurence of vacuum in Test 3), this is not surprising:
there hold minimum principles $\sigma_{xx}\ge (kh)^{-2},\sigma_{zz}\ge (kh)^2$ 
in the High-Weissenberg limit $\lambda\to\infty$, $\eta_p/\lambda=O(1)$, -- without topography -- where the relaxation source terms are neglected
(except that here initially $k=\infty$, but $k$ becomes hopefully finite after some time).
Hence a possible blow-up at the front, where $h\to0$, when $\lambda$ is too large for the diffusive relaxation source terms to compensate for the transport effects.
Note that in any case, the materials on the left of the contact-discontinuity wave (where $h$ is non-zero) is stretched in direction $\be_x$ and compressed in direction $\be_z$, 
except at the vacuum front where a small region with $\strs$ closer to the equilibrium $\I$ is constantly seen (may be artificially due to our numerical treatment of the vacuum front).
While no actual blow-up occurs, this will also be interpreted as a manifestation of a close-to-equilibrium stability property (see Section~\ref{sec:conclusion}).


\begin{figure}
 \include{ucm04590IH} 
 \include{ucm04590IU} 
\caption{\label{fig:test2rest1} Variations of the variables $h,u^0_{x}$ with $\eta_p, \lambda$ in Test case 2}
\end{figure}

\begin{figure}
 \include{ucm04590IX} 
 \include{ucm04590IZ}  
\caption{\label{fig:test2rest2} Variations of the variables $\sigma_{xx},\sigma_{zz}$ with $\eta_p, \lambda$ in Test case 2}
\end{figure}

{\bf Test case 3.} It is a benchmark introduced in~\cite{gallouet-herard-seguin-2003} to test the treatment of topography, see also~\cite{bouchut-2004}.
For $x\in(0,25)$, we compute until $T=.25$ the evolution from an initial condition
$(h,hu^0_x,h\sigma_{xx},h\sigma_{zz})(t=0,x)=((10-b)_+,-350+700H(x-50/3),(10-b)_+,(10-b)_+)$ over a topography $b(x)=H(x-25/3)-H(x-25/2)$.
Two rarefaction waves propagate on the left and right sides of the initial velocity singularity at $x=50/3$ 
so that a vacuum is created in between (in the usual Saint-Venant case).
In addition, a couple of rarefaction/shock waves is created at each singular point $x=25/3,25/2$
of the topography, but have much smaller amplitudes than the rarefaction waves at $x=50/3$.
The results in Fig.~\ref{fig:test3a1} and~\ref{fig:test3a2} are obtained for various  $\eta_p, \lambda$ at a constant $\eta_p/\lambda=10^{-4}$.
This particular choice was made because then, at the final time, the system is sufficiently close to the Saint-Venant limit 
$\eta_p/\lambda\to0$ so that the 
pressure is hardly modified compared with the usual Saint-Venant case.

Compared with the usual Saint-Venant case, the double rarefaction wave centered at $x=50/3$
cannot create vacuum but at the single point $x=50/3$ where the initial velocity is singular.
This can be explained as follows.
Assuming that the source terms in the stress equations do not influence much the bounds on $\sigma_{xx},\sigma_{zz}$,
in agreement with our Remark~\ref{Remark maxprinc}, the maximum (resp. minimum) 
principle holds for $s_{xx}$ (resp. $s_{zz}$), and there exists
a constant $k$ (depending only on the initial conditions since initially $h>0$) such that $\sigma_{xx}\ge (kh)^{-2},\sigma_{zz} \ge (kh)^{+2}$.
But according to the energy bound, one has that $(\eta_p/\lambda)\int h\sigma_{xx}dx$ remains bounded.
We deduce that $(\eta_p/\lambda)\int h^{-1}dx$ remains bounded, and therefore
$h$ cannot tend to $0$ on a whole interval, but can vanish on a single point.
We have then $\sigma_{xx}\to+\infty$ at the singular point, here $x=50/3$.


Moreover, another vacuum can be created at $x=25/2$ (of course still at a single-point because of the previous reasoning).
Of course, this new phenomena could just be a numerical artifact, for instance due to the computation of source terms by the hydrostatic reconstruction method.
In any case, assuming it is part of the richer phenomenology of our new model compared with the usual Saint-Venant model 
(see also Section~\ref{sec:conclusion} for possible physical interpretations), we observe that the existence of that phenomena depends on $\eta_p,\lambda$.
It happens only for $\eta_P=10^{-4},\lambda=10^{+0}$ and $\eta_P=10^{-3},\lambda=10^{+1}$ in our numerical experiments, not for $\eta_P=10^{-5},\lambda=10^{-1}$. 
But it occurs for larger $\lambda$ at $\eta_P=10^{-5}$ (when the stress relaxation terms are less important).
Indeed, this phenomena seems triggered mainly by high values of $\lambda$ (the ``High-Weissenberg limit'', which by the way requires very small time steps because of the CFL constraint),
and still holds at high values of $\eta_p$ (a ``rigid limit'' that seems to lead to some kind of a ``break-up'' of the jet here), 
even at very large $\eta_p$ when no vacuum occurs in between the rarefaction waves.
Notice that the latter phenomenon also induces  
an additional sign change for $u^0_x$ in between the two vacuum points, whose location 
seems to depend on $\lambda$ but not on $\eta_p$,
(another indication that the role of viscoelastic dissipation is dominant here).
And compared with the usual Saint-Venant case, velocities also assume much greater value (on the left in particular).

\begin{figure}
   \include{ucm16590IHd} 
   \include{ucm16590IUd} 
\caption{\label{fig:test3a1} Variations of the variable $h+b,u^0_x$ with $\eta_p, \lambda$ in Test case 3.
We use different labels for the positive ($\oplus$) and negative ($\ominus$) part of the velocity.
}
\end{figure}

\begin{figure}
   \include{ucm16590IXd} 
   \include{ucm16590IZd} 
\caption{\label{fig:test3a2} Variations of the variable $\sigma_{xx},\sigma_{zz}$ with $\eta_p, \lambda$ in Test case 3.}
\end{figure}

{\bf Test case 4.} In our last test case, we woud like to assess the treament of another type of topography source terms,
with creation of dry/wet fronts, by the hydrostatic reconstruction.
A usual test case is Thacker's~\cite{thacker-1981} e.g., which has analytical solutions in the usual Saint-Venant case.
But we could not capture interesting phenomena with Thacker's testcase, 
in particular because the CFL constraint requires the time step to go to $0$ very quickly
(on short time ranges) due to the creation of dry fronts where $h\to0$ and $\sigma_{xx}\to+\infty$ (with $\lambda$ not too small).
Note that this does not necessarily mean that this problem does not have global solutions with finite-energy.
A time-implicit scheme (probably hard to build) might be able to compute finite-energy approximations with a non-vanishing time-step.

Here, we consider the test case proposed by Synolakis~\cite{synolakis-1987} to model the runup of solitary waves.
This could be used until interesting final times $T=32.5$ after the incidental wave has reflected against the shore and created a dry front,
see also~\cite{marche-bonneton-fabrie-seguin-2007}.
We use $(h,hu^0_x,h\sigma_{xx},h\sigma_{zz})(t=0,x)= \left((1.+h_0(x)-b(x)\right)_+(1,\sqrt{g}h_0(x),1,1)$ 
as initial condition over a topography $b(x)=((x-40.)/19.85)_+$, $x\in(0,100)$.
The pertubation $h_0(x)=\alpha(\mathop{cosh}(\sqrt{.75\alpha}(x-\mathop{acosh}(\sqrt{1/.05})/(.75\alpha))))^{-2}$
models a solitary wave as a function of the parameter with $\alpha=.019/.1$
according to Synolakis semi-analytical theory.

The results in Fig.~\ref{fig:test4a} and~\ref{fig:test4b} show that 
it is essentially the variations of $\eta_p/\lambda$ that influence the water height and velocity 
among all possible variations of $\eta_p,\lambda$.
And although the first effect of the variations of $\eta_p/\lambda$ is on the waves celerity,
there is no direct match between variations in $\eta_p/\lambda$ and a time shift as shown in Fig.~\ref{fig:test4a}.
On the contrary, the variables $\sigma_{xx},\sigma_{zz}$ depend more on $\lambda$ alone 
(recall the importance of relaxation source terms, especially in the case where $h\to0$),
at least for such small values of $\eta_p/\lambda$ as those tested here
(sufficiently close to the Saint-Venant regime for the time step not to vanish, even at high values of $\lambda$).
The smaller $\lambda$ is, the stronger the dissipation is and thus balances the large stress values
that were induced close to the dry front. 

\begin{figure}
   \include{ucm13590IH} 
   \include{ucm13590IU} 
\caption{\label{fig:test4a} Variations of the variable $h+b,u^0_x$ with $\eta_p, \lambda$ in Test case 4}
\end{figure}

\begin{figure}
   \include{ucm13590IX} 
   \include{ucm13590IZ} 
\caption{\label{fig:test4b} Variations of the variable $\sigma_{xx},\sigma_{zz}$ with $\eta_p, \lambda$ in Test case 4}
\end{figure}

\section{Conclusion}
\label{sec:conclusion}

We have proposed a new reduced model for the motion of thin layers of viscoelastic fluids 
(shallow viscoelastic flows)
that are described by the upper-convected Maxwell model and driven by the gravity,
under a free surface and above a given topography with small slope 
(like in the standard Saint Venant model for shallow water).
More precisely,
we have shown formally that for given boundary conditions and under scaling assumptions (H1-5),
the solution to the incompressible Euler-UCM system of equations can be approximated by the solutions
to the reduced model~\eqref{eq:reduced} in some asymptotic regime.
We hope that this asymptotic regime in particular is physically meaningful, that and our new model makes sense, 
possibly beyond the previous asymptotic regime. (That is why we have studied it mathematically and explored it numerically
without constraining ourself to a particular regime, as it is usual in such cases.)

Observe that in the end we have obtained a flow model whose dynamics is function of the first normal stress difference only,
while the shear part of the stress is negligible and computed as an output of the flow evolution.
More specifically, the boundary conditions~(\ref{eq:notension}--\ref{eq:nofriction})
and the flat velocity profile (consequence of the assumed motion by slices)
require compatibility conditions on the bulk behaviour of $\tau_{xz}$ inside a thin layer.
Before looking in future works for other asymptotic regimes, 
possibly compatible (under different assumptions) with more general kinematics,
we would like to conclude here with a better insight of the physical implications of our reduced model.

\subsection{Physical interpretation from the macroscopic mechanical viewpoint}
\label{sec:interpretationmacro}

We note that the main differences between our model for shallow (Maxwell) viscoelastic flows
and the standard Saint Venant model for shallow water 
is
i) a new hydrostatic pressure (\ref{eq:p1}), which is function of the (viscoelastic) internal stresses in addition to the water level $h$, hence
ii) a new hydrodynamic force in the momentum balance (in addition to the external gravity force), which is proportional to 
the normal stress difference $\tau_{xx}-\tau_{zz}$, and
iii) variable internal stresses $\tau_{xx}$ and $\tau_{zz}$, which have their own dynamics
corresponding to a viscoelastic mechanical behaviour (with a finite relaxation time $\lambda=O(1)$ ;
such that one recovers the standard viscous mechanical behaviour only in the limit $\lambda\to0$).
Moreover, in the asymptotic regime where our non-Newtonian model was derived, with a small viscosity parameter $\eta_p =O(\e)$,
the strain and stress tensors have the scaling 
\begin{equation}
\label{eq:matrix-scaling}
\gbu = \begin{pmatrix}
        O(1) & O(\e) \\ O(\e) & O(1)
       \end{pmatrix},
\qquad
\str = \begin{pmatrix}
        O(\e) & O(\e^2) \\ O(\e^2) & O(\e)
       \end{pmatrix}.
\end{equation}
One essential rheological feature of our reduced model is thus the ratio $\e$ between the shear and elongational components of the stress tensor $\str$ and of the strain tensor $\gbu$.
The fact that our model should mainly describe extensional flows, with small shear (of the same small order as the elongational viscosity), 
seems to be  a strong limitation to the applicability of our model in real situations.
Of course, one is likely to need another reduced model (in other asymptotic regimes) to describe flows that are not essentially elongational.

Note yet that there are situations where physicists arrive at similar conclusions
\cite{entov-yarin-1984,entov2006dynamics} and obtain a very similar
one-dimensional model with purely elongational stresses for the description of free axisymmetric jets.
By the way, a description of free axisymmetric jets is also well achieved by our model
since the pure slip boundary conditions (\ref{eq:pureslip})-(\ref{eq:nofriction}) is equivalent to
assuming a cylindrical symmetry around the symmetry line of the jet,
and surface tension effects (neglected in our model) can be included using standard modifications
of our no-tension boundary condition (\ref{eq:notension}). 

Moreover, it seems possible to still include non-negligible shear effects in our model
through a parabolic correction of the vertical profile like in~\cite{gerbeau-perthame-2001,marche-2007},
as well as surface tension and friction effects of order one at the boundaries.

\subsection{Physical interpretation at the microscopic molecular level}
\label{sec:interpretationmicro}

A microscopic interpretation of our asymptotic regime
can also be achieved using a molecular model of the elastic effects
(that is, a model at the molecular level from which the UCM is a coarse-grained version
at the macroscopic mechanical level).
Following~\cite{bird-curtiss-armstrong-hassager-1987b},
a typical molecular model that accounts for the elasticity of a fluid 
invokes the transport of elastically deformable Brownian particles diluted in the fluid
(which can often be thought of as large massive molecules like polymers).
The simplest model of this kind couples, locally in the physical space, a kinetic theory for ``dumbbells''
(two point-masses connected by an elastic force idealized as a ``spring'') with the strain of the fluid.

Let us denote $\X_t(\x)$ the connector vector between the two point-masses
of a dumbbell modelling a polymer molecule at position $\x$ and time $t$ in the fluid.
The collection of vector stochastic processes $(\X_t(\x))_{t\in(0,+\infty)}$ 
parametrized by $\x\in\D_t$ is solution to overdamped Langevin equations
\begin{equation}
\label{eq:ito}
d \X_t  + (\bu\cdot\grad) \X_t dt
= \left( (\gbu) \X_t-\frac{2}{\zeta}\mathbf{F}(\X_t) \right) dt + 2\sqrt{\frac{k_B T}{\zeta}} d\mathbf{B}_t
\end{equation}
for a given field $(\mathbf{B}_t(\x))_{t\in(0,+\infty)}$ of standard Brownian motions (in It\^o sense) where
$\zeta$ is a friction parameter, $k_B$ the Boltzmann constant and $T$ the absolute thermodynamical temperature.
The UCM equations can be exactly recovered with the specific choice $\mathbf{F}(\X)=H\X$. 
Indeed, the extra-stress $\str$ and the conformation tensor $\strs$ 
are given by Kramers relation
\begin{equation}
\label{eq:stress}
\str=\frac{\eta_p}{2\lambda}(m\:\strs-\I)
\quad \text{ with } \quad 
\strs(t,\x)=\frac1H\mathbb{E}\Brk{\tilde\X_t\otimes\mathbf{F}(\tilde\X_t)}=\int\Brk{\tilde\X\otimes\tilde\X}\psi(t,\x,\ell\tilde\X)d\tilde\X
\end{equation}
where $\tilde\X=\X/\ell$ is an adimensional version of $\X$, $m=\frac{H\ell^2}{k_B T}$ is a ratio between the elastic potential energy and the heat of the Brownian bath,
and $\eta_p=2 \lambda n k_B T$ is the molecular interpretation of the polymer viscosity,
with $n$ the number density 
of polymer chains by unit volume (assumed constant as usual for dilute polymer solutions) and $\lambda$ a characteristic time for dumbbells.
One can always choose $\ell$ such that $m=1$.
Then, choosing $\lambda=\frac{\zeta}{4H}$ as a relaxation time, It\^o formula allows one to exactly recover the UCM system of equations~(\ref{eq:ucm}) 
when the solvent is assumed inviscid with a velocity field $\bu(t,\x)$ solution to the Euler equations, 
on noting that the probability density $\psi(t,\x,\ell\tilde\X)$ satisfies the following Fokker-Planck equation
on the unbounded domain $\tilde\X\in\R^2$:
\begin{equation}
\label{eq:fokker-planck}
\pd{\psi}{t} + \bu\cdot\grad\psi =
-\div_{\tilde\X}\brk{[(\gbu)\tilde\X-\frac{1}{2\lambda}\tilde\X]\psi} 
+ \frac{1}{2\lambda}\Delta_{\tilde\X}\psi \,. 
\end{equation}

We note that in~\cite{narbona-reina-bresch-2010} a reduced model for shallow viscoelastic flows
quite similar to ours has already been derived starting from a coupled micro-macro system like
(\ref{eq:fokker-planck}--\ref{eq:stress}--\ref{eq:ns}), rather than starting from a coarse-grained
system at the macroscopic level like the UCM model.
The difference between the Hookean micro-macro system above (equivalent in some sense to the UCM model)
and the micro-macro system used in~\cite{narbona-reina-bresch-2010} is the spring force:
it corresponds to \textit{FENE} dumbbells $\mathbf{F}(\X_t)=\X_t/(1-|\X_t|^2/b)$ in~\cite{narbona-reina-bresch-2010}.
The FENE force is more physical because it accounts for a finite extension $|\X_t|< b$, but contrary to the Hookean dumbbells,
it does not have an exact coarse-grained macroscopic equivalent like the UCM model.
Yet, if we follow the same procedure as in~\cite{narbona-reina-bresch-2010} but for Hookean dumbbells,
we can hope to derive a reduced micro-macro model whose coarse-grained version is comparable to our new reduced UCM model.
Moreover, if the scaling regimes are the same as in~\cite{narbona-reina-bresch-2010}, then
our model should also compare to that in~\cite{narbona-reina-bresch-2010},
for an inviscid solvent, in the infinite extensibility limit $b\to\infty$ 
(where one formally recovers the Hookean dumbbells from FENE dumbbells).
Now, observe that the scaling of our new model implies~\eqref{eq:matrix-scaling}
$\gbu=\boldsymbol{\gamma}_0 + O(\e)$ where $\boldsymbol{\gamma}_0 = O(1)$ is 
a traceless diagonal matrix with entries $\pds{x}u^0_x,-\pds{x}u^0_x$.
Then~\eqref{eq:fokker-planck} rewrites 
\begin{equation}
\label{eq:fokker-planckbis}
\pd{\psi}{t} + \bu\cdot\grad\psi =
\frac{1}{2\lambda} \div_{\X}\brk{M\grad_{\X}\brk{\frac\psi{M}}} + O(\e) \,,
\end{equation}
where $M(t,\x,\X)$ is a weight function proportional to the Maxwellian $e^{-\X^T(2\lambda\boldsymbol{\gamma}_0-\I)\X}$.
The approximation~\eqref{eq:fokker-planckbis} of~\eqref{eq:fokker-planck} is consistent with our new reduced model
provided it yields a consistent approximation for the stress in~\eqref{eq:stress}:
that is, it suffices to show $\sigma_{xx},\sigma_{zz}=O(1)$ and $\sigma_{xz}=O(\e)$ as $\e\to0$.
To this aim, let us define an order-one approximation $\psi^0=\psi+O(\e)$ solution to
\begin{equation}
\label{eq:fokker-planck0}
\pd{\psi^0}{t} + \bu^0\cdot\grad\psi^0 =
\frac{1}{2\lambda} \div_{\X}\brk{M\grad_{\X}\brk{\frac{\psi^0}{M}}} \,.
\end{equation}
The point is to estimate the terms
\begin{equation}
\label{eq:stress0}
\str^0=\frac{\eta_p}{2\lambda}(\strs^0-\I), 
\qquad \strs^0
=\int\Brk{\X\otimes\X}\psi^0(\X)d\X \,.
\end{equation}
This is not an easy task because of the coupling between $\psi^0$ and $\bu^0$. 
Yet it seems reasonable to assume that $\psi^0$ remains close to the equilibrium solution $M/\int M$ for all times
(indeed, the Hookean force is derived from an $\alpha$-convex potential~\cite{arnold-et-al-2004}),
and in particular the Maxwellian $\psi^0 \propto e^{- (Ax^2+Bz^2+2Cxz)}$ has the scaling $A=O(1),B=O(1),C=O(\e)$,
which implies that $\strs$ (and thus $\str$) is diagonal at first order.
One then obtains a reduced kinetic model which can be exactly coarse-grained into our new reduced UCM model with It\^o formula.

A macroscopic consequence of the microscopic assumption above is that the reduced model is well-adapted for elongational flows, which is consistent with our macroscopic intepretation of the model. 
Indeed, everywhere in the macroscopic physical space, one can only expect a balance of internal elastic energy due to stretching or compressing strains in the directions $\be_x$ and $\be_z$,
which is the case in elongational flows. (There is not a high probability of permanently sheared dumbbells.)
Moreover, if $\psi^0$ is actually close to the equilibrium $M$ (the particular case $A\approx2\lambda\pds{x}u^0_x-1,B\approx-2\lambda\pds{x}u^0_x-1,C\approx0$ of our assumption), then, at first-order, the dumbbells are quite uniformly oriented but stretched in one canonical direction -- $\be_x$ or $\be_z$ -- and necessarily compressed in the orthogonal one (the level-sets of the distribution function are ellipsoidal with principal axes $\be_x$ and $\be_z$ at first order).
This was indeed observed in those numerical experiments where no blow-up phenomenon seemed to occur.

The microscopic view is in turn a plausible physical explanation at the molecular level of some macroscopic observations.
Recall indeed that one-dimensional simple models similar to our model have already been derived in the past to model axisymmetric free jets of elastic liquids~\cite{entov-yarin-1984,entov2006dynamics} with a view to explaining the die swell at the end of an extrusion pipe.
Now, a miscroscopic interpretation of the die swell is: the elastic energy stored before the die 
is released after the die. The dumbbells, mainly compressed in the radial direction $\be_z$ before the die, stretch 
just after the die. This may be responsible for an increase of the jet radius (the free-surface of the jet flow equilibrates with the atmospheric pressure) after a characteristic relaxation time linked to $\lambda$, hence the so-called {\it delayed die swell}.

Finally, we would like to comment on the results obtained in~\cite{narbona-reina-bresch-2010} 
with FENE dumbbells. 
The main differences with our reduced model (which has the micro-macro interpretation explicited above) are:
(i) 
the relaxation time in~\cite{narbona-reina-bresch-2010} is assumed small $\lambda=O(\e)$, because then it is possible to compute approximate solutions to the Fokker-Planck equation
following the Chapman-Enskog procedure of~\cite{degond-lemou-picasso-2002} ; and  
(ii) 
the polymer distribution 
is mainly radial ($\psi^0$ is only function of $|\X|$), 
because the authors claim that this suffices to next imply $\sigma_{xz} = O(\e)$ and, as a consequence, a flat profile for the horizontal velocity like in our model.
Then, the scaling regimes are not the same, and the radial assumption is too strong to allow one to recover our ellipsoidal probability distribution.
So we cannot directly compare our results though they have a similar flavour.

\subsection{Open questions and perspectives}

First, regarding the interpretation of our model, one might ask
whether the present scaling corresponds to a physical situation actually observed for elastic fluids in nature.
In particular, the main questionable assumption is of course the pure-slip and no-friction 
boundary conditions~(\ref{eq:pureslip}--\ref{eq:nofriction}) at the bottom
(already unrealistic for Newtonian flows, maybe even more unrealistic for non-Newtonian ones).
Second, future works on this topic might consider the following directions:
\begin{itemize}
 \item derive thin-layer reduced models with other equations modelling non-Newtonian flows,
which are believed to better model the rheological properties of real materials
(constitutive models like Giesekus, PTT, FENE-P, or other molecular models 
than the FENE dumbbell model used in~\cite{chupin-2009b,narbona-reina-bresch-2010}),
and in two-dimensional settings (see~\cite{bouchut-westdickenberg-2004,marche-2007} for the standard shallow water model);
 \item derive a reduced model closer to real physical situations, possibly in different regimes, or
for instance by using a $z$-dependent velocity profile $u_x$  (possibly a multi-layer model)
and different boundary conditions than~(\ref{eq:notension}) and~(\ref{eq:nofriction}) 
(with surface tension and friction at the bottom),
which may lead to find physical regimes where $\tau_{xz}$ is not negligible;
 \item give a rigorous mathematical meaning and enhance numerical simulations (well-balanced second-order reconstructions)
for non-standard systems of equations like the new one presented here.
\end{itemize}
We note that multi-layer models are also a path to the modelling of some important physical situations,
like a thin layer of polymeric fluids on water to forecast the efficiency of oil slick protection plans.

\appendix

\section{Convexity of the energy}
\label{app:convexity}

In order to check the convexity of $\widetilde E$ in \eqref{eq:energy} with respect to general variables,
we use a Lagrange transformation, see for example Lemma 1.4 in~\cite{bouchut-2004}.
Thus $\widetilde E$ is a convex function of
$$
\left(h,hu^0_x,h\varpi^{-1}\left(\frac{\sigma_{xx}^{-1/2}}{h}\right),h\varsigma^{-1}\left(\frac{\sigma_{zz}^{1/2}}{h}\right)\right)
$$
for given smooth invertible functions $\varpi$, $\varsigma$, if and only if $\widetilde E/h$ is a convex function of the Lagrangian variables 
$$
V = \left(\frac1h,u^0_x,\varpi^{-1}\left(\frac{\sigma_{xx}^{-1/2}}{h}\right),\varsigma^{-1}\left(\frac{\sigma_{zz}^{1/2}}{h}\right)\right)
\,.
$$
Let us denote by $V_i$, $i=1,\ldots,4$ the entries of the vector $V$,
then the Lagrangian energy writes
$$
\frac{\widetilde E}{h} = \frac12 V_{{2}}^{2} + \frac{g}2 \frac1{V_1}+gb
 + \frac{\eta_{p}}{4\lambda} \left( 
      \frac{V_1^2}{\varpi\left(V_3\right)^{2}}
      + \frac{\varsigma\left( V_4 \right)^{2}}{V_1^2}
      - \ln\left( \frac{\varsigma\left(V_4\right)^{2}}{\varpi\left(V_3\right)^{2}} \right)-2
  \right).
$$
Introduce now the notation
$$
	\Omega(V_3)=2\ln\varpi(V_3),\qquad
	\zeta(V_4)=-2\ln\varsigma(V_4).
$$
Clearly we only need to look at the convexity with respect to $(V_1,V_3,V_4)$,
and the Hessian matrix $\mathcal{H}$ of $\widetilde E/h$ with respect to these variables (at fixed $b$) is given by
$$ \frac{4\lambda}{\eta_p}\mathcal{H}=\left[ 
\begin {array}{ccc}
\frac{4\lambda g}{\eta_p}\frac{1}{V_1^3} +  2e^{-\Omega} + \frac{6e^{-\zeta}}{V_1^4}
&
-2 V_1e^{-\Omega}\Omega'
&
2\frac{e^{-\zeta}\zeta'}{V_1^3}
\\
-2V_1e^{-\Omega}\Omega'
&
    V_1^2e^{-\Omega}\bigl(\Omega'^2-\Omega''\bigr)+ \Omega''
&
0
\\ 
2\frac{e^{-\zeta}\zeta'}{V_1^3}
&
0
&
  \frac{e^{-\zeta}}{V_1^2}\bigl(\zeta'^2-\zeta''\bigr)+\zeta''
\end {array} \right]
\,,
$$
where prime denotes the derivative with respect to the involved $V_i$.
Since $V_1$ can take any positive value at fixed $V_3$ or $V_4$,
the positivity of the diagonal terms give the necessary conditions
$$
0<\Omega''(V_3)<\Omega'(V_3)^2,
\qquad
0<\zeta''(V_4)<\zeta'(V_4)^2.
$$
Then, writing the positivity of the determinant of the $2\times 2$ upper left
submatrix of $\mathcal{H}$, and looking at the dominant term when
$V_1\rightarrow\infty$ yields the necessary condition
$$2e^{-2\Omega}(\Omega'^2-\Omega'')-4e^{-2\Omega}\Omega'^2>0.$$
Obviously there is no function $\Omega(V_3)$ satisfying these conditions,
and $\widetilde E$ is never convex with respect to the considered variables.

On the contrary, if we choose the physically natural, but non-conservative,
variables $q=(h,hu_x^0,h\sigma_{xx},h\sigma_{zz})$, then using the Lagrangian variables 
$$
W = \left(\frac1h,u^0_x,\sigma_{xx},\sigma_{zz}\right) \,,
$$
one can write
$$
\frac{\widetilde E}{h} = \frac{(u_x^0)^{2}}2 + \frac{gh}2 +gb + \frac{\eta_{p}}{4\lambda}
 \left( \sigma_{xx} + \sigma_{zz} - \ln\left(\sigma_{xx}\sigma_{zz}\right)-2 \right),
$$
which is obviously convex with respect to $W$ (at fixed $b$). We conclude that $\widetilde E$
is convex with respect to $q$.

\bibliographystyle{amsplain}

\providecommand{\bysame}{\leavevmode\hbox to3em{\hrulefill}\thinspace}
\providecommand{\MR}{\relax\ifhmode\unskip\space\fi MR }
\providecommand{\MRhref}[2]{%
  \href{http://www.ams.org/mathscinet-getitem?mr=#1}{#2}
}
\providecommand{\href}[2]{#2}

\end{document}